\documentclass[11pt, a4paper]{article}
\usepackage{amsmath,amssymb,amsthm,enumerate,graphicx,mathtools,titling}
\usepackage[margin=2.5cm]{geometry}

\usepackage{xypic} 
\usepackage{mathrsfs} 
\usepackage{xspace}
\usepackage{tikz-cd}
\usepackage[margin=2.5cm]{geometry}
\usepackage{hyperref}

\newtheorem{thm}{Theorem}[section]

\newtheorem{prop}[thm]{Proposition}
\newtheorem{lem}[thm]{Lemma}
\newtheorem{cor}[thm]{Corollary}

\theoremstyle{definition}
\newtheorem{defn}[thm]{Definition}
\newtheorem{quest}[thm]{Question}
\newtheorem{prob}[thm]{Problem}
\newtheorem{rmk}[thm]{Remark}

\newtheorem*{claim*}{Claim}

\newtheorem{example}[thm]{Example}

\newcommand{\mc}[1]{\mathcal{#1}}

\newcommand{\ms}[1]{\mathscr{#1}} 

\newcommand{\Zb}{\mathbb{Z}}

\newcommand{\Fb}{\mathbb{F}}
\newcommand{\Nb}{\mathbb{N}}
\newcommand{\Rb}{\mathbb{R}}
\newcommand{\Qb}{\mathbb{Q}}

\newcommand{\tdlc}{t.d.l.c.\@\xspace}

\newcommand{\lcsc}{l.c.s.c.\@\xspace}
\newcommand{\defbold}{\textbf}

\newcommand{\inv}{^{-1}}
\newcommand{\triv}{\{1\}}

\newcommand{\tC}{(\mathrm{C})}
\newcommand{\tD}{(\mathrm{D})}
\newcommand{\tI}{(\mathrm{I})}
\newcommand{\tCDI}{(\mathrm{CDI})}

\newcommand{\Homeo}{\mathrm{Homeo}}

\newcommand{\CC}{\mathrm{C}}
\newcommand{\N}{\mathrm{N}}

\newcommand{\Sym}{\mathrm{Sym}}
\newcommand{\Alt}{\mathrm{Alt}}
\newcommand{\Aut}{\mathrm{Aut}}
\newcommand{\Mon}{\mathrm{Mon}}
\newcommand{\Inn}{\mathrm{Inn}}

\newcommand{\Res}{\mathrm{Res}}

\newcommand{\Sub}{\mathrm{Sub}}

\newcommand{\RadRE}{\mathrm{Rad}_{\mc{RE}}}

\newcommand{\cgrp}[1]{\overline{\langle #1 \rangle}}

\newcommand{\grp}[1]{\langle #1 \rangle}
\newcommand{\ol}[1]{\overline{#1}}

\newcommand{\propP}[1]{(\mathrm{P}_{#1})}

\begin{document}

\title{Decomposition of locally compact coset spaces}

\preauthor{\large}
\DeclareRobustCommand{\authoring}{
\renewcommand{\thefootnote}{\arabic{footnote}}
\begin{center}Colin D. Reid\textsuperscript{1}\footnotetext[1]{Research supported by ARC grant FL170100032.}
\\ \bigskip
The University of Newcastle, School of Mathematical and Physical Sciences, Callaghan, NSW 2308, Australia.\\
\href{mailto:colin@reidit.net}{colin@reidit.net}
\end{center}
}
\author{\authoring}
\postauthor{\par}

\maketitle
\begin{abstract}
In \cite{RW-EC} it was shown that a compactly generated locally compact group $G$ admits a finite normal series $(G_i)$ in which the factors are compact, discrete or irreducible in the sense that no closed normal subgroup of $G$ lies properly between $G_{i-1}$ and $G_{i}$.  In the present article, we generalize this series to an analogous decomposition of the coset space $G/H$ with respect to closed subgroups, where $G$ is locally compact and $H$ is compactly generated.  This time, the irreducible factors are coset spaces $G_{i}/G_{i-1}$ where $G_{i}$ is compactly generated and there is no closed subgroup properly between $G_{i-1}$ and $G_{i}$.  Such irreducible coset spaces can be thought of as a generalization of primitive actions of compactly generated locally compact groups; we establish some basic properties and discuss some sources of examples.
\end{abstract}

\tableofcontents

\addtocontents{toc}{\protect\setcounter{tocdepth}{1}}

\section{Introduction}

Let $G$ be a locally compact second-countable (\lcsc) group.  In \cite{RW-EC}, it was shown that if $G$ is compactly generated, then $G$ admits an \defbold{essentially chief series}, meaning a finite series
\[
\triv = G_0 < G_1 < \dots < G_n = G
\]
of closed normal subgroups $G_i$, such that each factor $G_{i+1}/G_i$ is compact, discrete or a \defbold{chief factor} of $G$, that is, with no closed normal subgroup of $G$ lying properly between $G_i$ and $G_{i+1}$.

We now propose an analogous factorization in a more general setting.  Let $G$ be a group and let $X$ be a locally compact $G$-space, that is, a locally compact Hausdorff space on which $G$ acts by homeomorphisms.  Given a $G$-space $X$, we define a \defbold{$G$-kernel} to be a closed $G$-invariant equivalence relation $R$, and then a \defbold{$G$-factor of $X$} is the corresponding quotient map $X \rightarrow X/R$ (or just the space $X/R$, if we refer to properties of $G$-spaces).  We refer to the identity relation as the trivial $G$-kernel.  We say that a $G$-factor $\pi: X \rightarrow X/R$ (or the corresponding $G$-kernel $R$) is
\begin{itemize}
\item Type $\tC$ if there is an open cover $\mc{O}_R$ of $X/R$ such that $\pi\inv(O)$ has compact closure for all $O \in \mc{O}_R$;
\item Type $\tD$ if there is an open cover $\mc{O}$ of $X$ such that the restriction of $\pi$ to $O$ is injective for all $O \in \mc{O}$;
\item Type $\tI$ if $R$ is a nontrivial $G$-kernel, but the only $G$-kernel properly contained in $R$ is the trivial one.  In other words, whenever we have $G$-factors
\[
X \xrightarrow{\pi_1} X/R_1 \xrightarrow{\pi_2} X/R,
\]
such that $\pi = \pi_2\pi_1$, then one of $\pi_1$ and $\pi_2$ is injective.
\end{itemize}

A \defbold{(finite) factorization} of $X$ is a finite sequence of factors
\[
X \xrightarrow{\pi_1} X/R_1 \xrightarrow{\pi_2} \dots \xrightarrow{\pi_n} X/R_n
\]
such that $R_n = X \times X$, in other words, $X/R_n$ is a singleton.  We say that the factorization is a \defbold{$\tCDI$ factorization} if for $1 \le i \le n$, $\pi_i$ is of type $\tC$, $\tD$ or $\tI$.

We now have the following general problem.

\begin{prob}
Determine sufficient conditions, given a locally compact Hausdorff space $X$ and a group $G \le \Homeo(X)$, for $X$ to admit a $\tCDI$ factorization.
\end{prob}

For this article we will specialize to the case of coset spaces, that is, $G$ is a topological group and $X = G/H$ for some closed subgroup $H$ of $G$.  In this context, the factorization corresponds to a sequence
\[
H = H_0 < H_1 < \dots < H_n = G
\]
of closed overgroups of $H$, and the three special types of factor are as follows:
\begin{itemize}
\item $H_{i+1}/H_i$ is of type $\tC$ if and only if $H_i$ is cocompact in $H_{i+1}$;
\item $H_{i+1}/H_i$ is of type $\tD$ if and only if $H_i$ is open in $H_{i+1}$;
\item $H_{i+1}/H_i$ is of type $\tI$ if and only if $H_i$ is maximal among proper closed subgroups of $H_{i+1}$.
\end{itemize}

The existence of an essentially chief series for a group $G$ corresponds exactly to the existence of a $\tCDI$ factorization for the coset space $G^2/\Delta$, where $G^2 = G \times G$ and $\Delta = \{(g,g) \mid g \in G\}$; see Proposition~\ref{prop:diagonal}.  Our main theorem is that a $\tCDI$ factorization exists for a more general class of locally compact coset spaces.

\begin{thm}[{See Section~\ref{sec:connected}}]\label{thm:cdi}
Let $G$ be a locally compact group, let $H$ be a compactly generated closed subgroup of $G$, let $O$ be a compact neighbourhood of the identity in $G$ and let $R$ be a compact normal subgroup of $G$ such that $G^\circ R/R$ is a Lie group.  Then there is a series
\[
H \le HR = H_0 \le H_1 \le \dots \le H_n = G
\]
of closed subgroups of $G$ such that for each $1 \le i \le n$, at least one of the following holds:
\begin{enumerate}[(i)]
\item $H_{i-1} \cap O = H_i \cap O$ and $H_{i-1} = \grp{H,H_{i-1} \cap O}$;
\item $H_i \subseteq OH_{i-1}$ and both $H_{i-1}$ and $H_i$ are compactly generated;
\item $H_{i-1}$ is a maximal proper closed subgroup of $H_i$ and $H_i = \grp{H,H_{i} \cap O}$.
\end{enumerate}
\end{thm}

Say that $A/B$ is an \defbold{irreducible coset space} if $B$ is maximal among proper closed subgroups of $A$.  Theorem~\ref{thm:cdi} shows that irreducible coset spaces play an important role in general coset spaces in locally compact groups, and more specifically, starting from a locally compact coset space $G/H$ where $H$ is compactly generated, the relevant irreducible coset spaces are of the form $A/B$ where $A$ is compactly generated.  We can regard irreducible coset spaces as a generalization of primitive continuous actions of topological groups, recalling that a transitive group action is primitive if and only if each point stabilizer is abstractly a maximal subgroup (not just maximal among proper closed subgroups).  This additional generality brings complications compared to the primitive case, and it should be said that even primitive continuous actions of compactly generated locally compact groups are far from classified.  However, we are able to establish some basic properties of irreducible coset spaces; see Section~\ref{sec:irreducible}.

In understanding the structure of an irreducible coset space $G/H$, there is little lost in assuming that it is a faithful $G$-space, that is, $\bigcap_{g \in G}gHg\inv = \triv$.  Write $\mc{M}_G$ for the set of $M \le G$ such that $M$ is minimal among the nontrivial closed normal subgroups of $G$.  The following theorem summarizes some of the basic features of faithful irreducible coset spaces $G/H$, in the case where $G$ is a compactly generated locally compact group.

\begin{thm}[{See Section~\ref{sec:cg_irreducible}}]\label{thm:cg_irreducible}
Let $G$ be a compactly generated locally compact group and let $H \le G$ be such that $G/H$ is a faithful irreducible coset space.  Then $G$ is second-countable and $\CC_H(N) = \triv$ for every nontrivial normal subgroup $N$ of $G$.  Moreover, exactly one of the following holds:
\begin{enumerate}[(i)]
\item There is a compact $M \in \mc{M}_G$, such that $G$ splits as a semidirect product $M \rtimes H$, and there exists a continuous injective homomorphism from $\CC_G(M)$ to $M$ with dense image.
\item We have $\mc{M}_G \neq \emptyset$, but $G$ has no nontrivial compact normal subgroups.  For every $M \in \mc{M}_G$ we have $M \cap H = \triv$, but $\CC_G(M) \neq \triv$; moreover, the coset space $(M \rtimes H)/H$ is also faithful and irreducible.
\item There is a unique $M \in \mc{M}_G$, which is the intersection of all nontrivial closed normal subgroups of $G$, and $\CC_G(M) = \triv$.
\item $\mc{M}_G = \emptyset$ and every nontrivial closed normal subgroup contains a descending chain of nontrivial discrete normal subgroups with trivial intersection.
\end{enumerate}
\end{thm}

Case (i) is well-understood, including the classification of the compact groups $M$ that can occur.  It is easy to see that every group $M$ that appears as a chief factor in a compactly generated locally compact group will also appear as $M \in \mc{M}_G$ for some irreducible coset space $G/H$ satisfying case (i) or (ii) of Theorem~\ref{thm:cg_irreducible}.  It is more difficult to characterize the (necessarily nonabelian) chief factors that can occur as $M \in \mc{M}_G$ in case (iii) of the theorem; certainly there are examples where $M$ is of semisimple type, as we discuss in Section~\ref{sec:semisimple}.  Note that case (iii) is the only case in which $H$ can have nontrivial intersection with some $M \in \mc{M}_G$.  Case (iv) of the theorem is difficult to tackle further with present methods, even to find examples.  In all cases except case (i), more examples could help shed light on the general situation; of particular interest would be more examples where there is a nontrivial closed normal subgroup $N$ such that $NH < G$.

\section{Preliminaries}

\subsection{Coset spaces}

In this article we will focus on transitive actions $(X,G)$ on locally compact spaces.  To avoid some technicalities it is useful to assume that in fact $(X,G)$ is a \defbold{coset space}, that is, $G$ is a Hausdorff topological group and $X = G/H$ for some closed subgroup $H$ of $G$.  Unless otherwise specified, the action on a coset space $G/H$ is always the left translation action of $G$.  Since it makes no difference for the structure of the action, if $K$ is a closed $G$-invariant subgroup of $H$, we will identify $(G/K)/(H/K)$ with $G/H$ where there is no ambiguity in doing so.

We write $\Sub(G)$ for the set of closed subgroups of the topological group $G$; $\Sub(G)^G$ for the set of $G$-fixed points of $\Sub(G)$ under conjugation, that is, the closed normal subgroups of $G$; and $\Sub(G/H)$ for the set of $K \in \Sub(G)$ such that $K \ge H$.

Coset spaces account for most of the ``natural'' occurrences of transitive topological group actions, as the following lemma indicates.

\begin{lem}[{See \cite[Theorem~2.2.2]{BeckerKechris}}]\label{lem:transitive_topology}
Let $G$ be a Polish group acting continuously and transitively on the separable metrizable space $X$ and let $x \in X$.  Then the map $gG_x \mapsto g.x$ is a homeomorphism from $G/G_x$ to $X$.
\end{lem}

The following is an outline of how the decomposition of coset spaces (where we consider all closed subgroups containing a given closed subgroup) includes as a special case results about closed normal series in topological groups.

\begin{prop}\label{prop:diagonal}
Let $G$ be a Hausdorff topological group, write $G^2 = G \times G$, let $\Delta$ be the diagonal subgroup $\{(g,g) \mid g \in G\}$ and let $\iota: G \rightarrow G \times G; \; g \mapsto (g,1)$.
\begin{enumerate}[(i)]
\item There is an order-preserving bijection $\theta$ from $\Sub(G)^G$ to $\Sub(G^2/\Delta)$, given as follows: given $N \in \Sub(G)^G$, the corresponding element of $\Sub(G^2/\Delta)$ is $\theta(N) = \iota(N)\Delta$, and given $H \in \Sub(G^2/\Delta)$, the corresponding element of $\Sub(G)^G$ is $\theta\inv(H) = \iota\inv(H)$.
\item Let $M, N \in \Sub(G)^G$ such that $M \ge N$.  Then as topological $G$-spaces, $M/N$ is equivalent to $\theta(M)/\theta(N)$, where $G$ acts on $M/N$ by conjugation and $g \in G$ acts on $\theta(M)/\theta(N)$ as translation on the left by $(g,g)$.
\end{enumerate}
\end{prop}

\begin{proof}
(i)
Let $N \in \Sub(G)^G$.  Since $N$ is normal in $G$, we see that $\iota(N)$ is normal in $G^2$ and hence $H := \iota(N)\Delta$ is a subgroup of $G^2$ containing $\Delta$.  Given a net $(n_ig_i,g_i)$ of elements of $H$ converging to some $(a,b) \in G^2$, we see that $g_i \rightarrow b$ and hence $n_i \rightarrow ab\inv$, so $ab\inv \in N$ and hence $(a,b) \in H$, showing that $H$ is closed and hence $H \in \Sub(G^2/\Delta)$.  Moreover, we see that $\iota\inv(H) = N$.

Conversely, let $H \in \Sub(G^2/\Delta)$ and set $N:= \iota\inv(H)$.  Given $(a,b) \in H$, we can decompose $(a,b)$ uniquely as $(ab\inv,1)(b,b)$, where $(ab\inv,1) \in \iota(G)$ and $(b,b) \in \Delta$; thus $H = \iota(N)\Delta$.  Since $H$ is a closed subgroup of $G^2$, we see that $N$ is a closed subgroup of $G$.  Given $n \in N$ and $g \in G$, then 
\[
(gng\inv,1) = (g,g)(n,1)(g\inv,g\inv) \in H,
\]
 so $gng\inv \in N$.  Thus $N \in \Sub(G)^G$.  This completes the proof that $\theta$ is a bijection from $\Sub(G)^G$ to $\Sub(G^2/\Delta)$.  It is then clear that $\theta$ is order-preserving.

(ii)
Let $\pi: G^2 \rightarrow G^2/\Delta$ be the quotient map; it is easily seen that $\pi \circ \iota$ is a homeomorphism.

We define a map $\phi: M/N \rightarrow \theta(M)/\theta(N)$ by setting $\phi(mN) = \iota(mN)\Delta$.  Given the proof of (i), it is clear that $\phi$ is continuous and bijective, and from the fact that $\pi \circ \iota$ is a homeomorphism, we see that $\phi$ is also a homeomorphism.  For the equivalence of the $G$-actions, consider $g \in G$ and $m \in M$.  Then
\[
(g,g)\iota(mN)\Delta = \iota(gmN)(1,g)\Delta = \iota(gmN)(g\inv,1)\Delta = \iota(gmg\inv N)\Delta,
\]
so left translation by $(g,g)$ on $\theta(M)/\theta(N)$ is equivalent to conjugation by $g$ on $M/N$.
\end{proof}

\subsection{Preliminaries on locally compact groups}

We note the following relationship between compactly generated locally compact groups and their finitely generated subgroups.

\begin{lem}\label{lem:comp_gen:approximation}
Let $G$ be a compactly generated locally compact group, let $U$ be a compact identity neighbourhood in $G$ and let $H$ be a subgroup of $G$ (not necessarily closed) such that $G = UH$.  Then there is a finite subset $F$ of $H$ such that
\[
G = U\grp{F}.
\]
\end{lem}

\begin{proof}
Let $X$ be a compact symmetric generating set for $G$ containing the identity and let $Y = \overline{U}X$.  By continuity, $Y$ and $Y^2$ are compact.  Since $G = UH$ it follows that $Y^2 \subseteq UF$ for some finite subset $F$ of $H$.  Thus $Y^2\grp{F} = Y\grp{F} = U\grp{F}$, and by induction on $n$, we have $Y^n\grp{F} = U\grp{F}$ for all $n \in \Nb$.  Since $Y$ contains $X$, and $X$ generates $G$ as a semigroup, we conclude that $G = Y\grp{F}$.
\end{proof}

A locally compact group $G$ is \defbold{regionally elliptic} if every compactly generated closed subgroup of $G$ is compact.  (Equivalently, by Lemma~\ref{lem:comp_gen:approximation}, $G$ is regionally elliptic if every finite subset is contained in a compact subgroup.) The \defbold{regionally elliptic radical}, denoted by $\RadRE(G)$, is the union of all closed normal regionally elliptic subgroups of $G$. 

\begin{thm}[Platonov, \cite{Plat66}]\label{thm:platonov_radical}
For $G$ a locally compact group, $\RadRE(G)$ is the unique largest regionally elliptic closed normal subgroup of $G$, and 
\[
\RadRE(G/\RadRE(G)) = \triv.
\] 
\end{thm}

Write $G^\circ$ for the connected component of the identity in the topological group $G$.  Note that when $G$ is locally compact, the quotient $G/G^\circ$ is a totally disconnected locally compact (\tdlc) group.  We say a locally compact group $G$ is \defbold{almost connected} if $G/G^\circ$ is compact (equivalently, profinite).

A (real) \defbold{Lie group} is a topological group that is a finite-dimensional analytic manifold over $\Rb$ such that the group operations are analytic maps.  A Lie group $G$ can have any number of connected components, but $G^\circ$ is always an \emph{open} subgroup of $G$. The group $G/G^\circ$ of components is thus discrete.

We say a locally compact group $G$ is \defbold{finite-dimensional}, and define $\dim_{\Rb}(G) = \dim_{\Rb}(G^{\circ})$, if $G^\circ$ is a Lie group.

By Van Dantzig's theorem, a \tdlc group has arbitrarily small compact open subgroups.  The following is an easy consequence. 

\begin{lem}\label{lem:vD}
Let $G$ be a locally compact group.  Then $G$ has an almost connected open subgroup $U$; moreover, $U^\circ = G^\circ$, and $U$ is unique up to finite index.
\end{lem}

The Gleason--Yamabe theorem in turn relates almost connected groups to Lie groups; here is the version we will use.

\begin{thm}[Gleason--Yamabe; see {\cite[Theorem 4.6]{MZ}}]\label{thm:yamabe_radical}
Let $G$ be an almost connected locally compact group.  Then $\RadRE(G)$ is compact, and the quotient $G/\RadRE(G)$ is a Lie group with finitely many connected components.  Moreover, every identity neighbourhood in $G$ contains a compact normal subgroup $N$ such that $G/N$ is a Lie group with finitely many connected components.
\end{thm}

\begin{cor}\label{cor:second_countable}
Let $G$ be a $\sigma$-compact locally compact group without arbitrarily small nontrivial compact normal subgroups.  Then $G$ is second-countable.
\end{cor}

\begin{proof}
Let $U$ be an almost connected open subgroup of $G$ and let $O$ be an identity neighbourhood containing no nontrivial compact normal subgroups.  Since $G$ is $\sigma$-compact we have $G = UC$ where $C$ is countable.  Let $F$ be a finite subset of $C$.  By Theorem~\ref{thm:yamabe_radical} there is a compact normal subgroup $N_F$ of $U$ such that $U/N_F$ is a Lie group with finitely many connected components and $N_F \subseteq \bigcap_{g \in F}gOg\inv$; since $N$ is normal in $U$, in fact
\[
N_F \subseteq \bigcap_{g \in UF}gOg\inv.
\]
The condition on $O$ ensures that the compact normal subgroups $N_F$ have trivial intersection as $F$ ranges over the finite subsets of $C$.  Thus $U$ is an inverse limit of countably many Lie groups, showing that $U$ is second-countable.  In turn, since $G$ is partitioned into the countably many left cosets of $U$, we conclude that $G$ is second-countable.
\end{proof}

The set $\Sub(G)$ of closed subgroups of a locally compact group $G$ carries a natural compact Hausdorff topology, the \defbold{Chabauty topology}, with basic open neighbourhoods of $H \in \Sub(G)$ as follows, where $K$ ranges over the compact subsets and $U$ ranges over the nonempty open subsets of $G$ (see for example \cite[Chapitre 8, \S5]{Bourbaki}):
\[
V_{K,U}(H) := \{J \in \Sub(G) \mid J \cap K \subseteq HU \text{ and } H \cap K \subseteq JU\}.
\]
We will not need to use the Chabauty topology directly in this article, but it is useful to have in mind as another source of intuition on coset spaces in locally compact groups.

The Chabauty space of a locally compact group is also a natural example of an ordered topological space, with the partial order given by inclusion of subgroups.  Our investigation in the present article of decompositions of the coset space $G/H$ can also be interpreted as structure theory of the closed up-set $\Sub(G/H) := \{K \in \Sub(G) \mid K \ge H\}$, which is then itself a compact ordered topological space.

\section{The $\tCDI$ factorization of locally compact coset spaces}

\subsection{Thickenings}

\begin{defn}
Let $G$ be a group, let $O$ and $S$ be subsets and let $H$ be a subgroup of $G$.  We say $H$ is \defbold{$(O,S)$-reduced} if $H = \grp{S,H \cap O}$.

Given a subgroup $U$ of $G$, the \defbold{$U$-thickening} of $H$ is the set
\[
T_U(H) = \bigcap_{h \in H} h(UH \cap HU)h\inv.
\]
We say $H$ is \defbold{$U$-thick} if $T_U(H) = H$.
\end{defn}

\begin{lem}[See {\cite[Corollary~1.20]{HilgertNeeb}}]\label{lem:compact_submonoid}
Let $G$ be a Hausdorff topological group, let $M$ be a closed submonoid of $G$ and let $N = M \cap M\inv$.  Suppose that $M/N$ is compact.  Then $M=N$ and $M$ is a subgroup of $G$.
\end{lem}

\begin{prop}\label{prop:thick}
Let $G$ be a Hausdorff topological group, let $H$ be a closed subgroup of $G$ and let $U$ be a subgroup of $G$ such that $UH/H$ is compact.  Write $T = T_{U}(H)$.
\begin{enumerate}[(i)]
\item We have
\[
T = \bigcap_{h \in H}hUH = \bigcap_{h \in H}HUh.
\]
\item $T$ is the largest closed subgroup of $G$ such that $H \le T \subseteq UH$.
\item $T$ is $U$-thick.
\item Let $K$ be a closed subgroup of $G$ such that $H \le K$ and let $V = U \cap K$.  Then $T_V(H) = T_U(H) \cap K$.
\item Let $S$ be a subset of $G$.  If $H$ is $(U,S)$-reduced, then so is $T$.
\end{enumerate}
\end{prop}

\begin{proof}
Let $\mc{W} = \{HUh \mid h \in H\}$ and let 
\[
W = \{g \in G \mid Y \in \mc{W} \Rightarrow Yg \in \mc{W}\}.
\]
We see that $1 \in H \subseteq W$ and $W$ is closed under multiplication; thus $W$ is a submonoid of $G$.

We claim next that $W$ is a closed subgroup of $G$; by Lemma~\ref{lem:compact_submonoid} it is enough to show that $W$ is closed and $W/H$ is compact.  More specifically, it will suffice to show that 
\[
W = \bigcap_{h \in H}hUH;
\]
Let $g \in W$.  Then $HUg \in \mc{W}$, so $HUg = HUh\inv$ for some $h \in H$, and hence $HU = HU h\inv g\inv$.  In particular, since $1 \in HUh\inv$, we have $g\inv \in HU$, or in other words $g \in UH$.  Thus $W \subseteq UH$.  In fact, since $W$ is a monoid containing the group $H$, it follows that $W \subseteq hUH$ for all $h \in H$.  Conversely, let $g \in G$ be such that $g \in hUH$ for all $h \in H$.  Given $h \in H$, we can write $g = h\inv uk$ for some $u \in U$ and $k \in H$.  It follows that
\[
HUh g = HUuk = HUk \in \mc{W};
\]
thus $Y \in \mc{W} \Rightarrow Yg \in \mc{W}$.  Thus $g \in W$, proving that $W$ is the claimed intersection.

Part (i) now follows by observing that
\[
T = \bigcap_{h \in H} h(UH \cap HU)h\inv = \bigcap_{h \in H}(hUH \cap HUh\inv) = W \cap W\inv = W.
\]

We have just seen that $T$ is a closed subgroup of $G$.  Consider a subgroup $H'$ of $G$ such that $H \le H' \subseteq UH$.  Then $H'$ is invariant under left translation by $h$, so $H' \subseteq hUH$ for all $h \in H$.  Thus $H' \le T$; this completes the proof of (ii).

We observe that $UT = UH$, so $T$ is the largest closed subgroup of $G$ such that $H \le T \subseteq UT$.  Thus $T$ is $U$-thick, proving (iii).

For (iv), let $T' = T_{U}(H) \cap K$.  Clearly $T_{V}(H)$ is contained in both $T_{U}(H)$ and $K$, so $T_{V}(H) \le T'$.  On the other hand, we see that $T'$ is a closed subgroup of $G$ such that $H \le T' \subseteq UH \cap K$: since $H \le K$, in fact $UH \cap K = (U \cap H)K = VK$.  Since $T_V(H)$ is the largest closed subgroup between $H$ and $VH$, it follows that $T' \le T_{V}(H)$, and hence equality holds.

For (v), let $S$ be a subset of $G$ and suppose that $H$ is $(U,S)$-reduced.  Since $H \le T \subseteq UH$, we have $T = (U \cap T)H$; hence
\[
T = \grp{U \cap T,H} = \grp{U \cap T,U \cap H,S} = \grp{U \cap T, S},  
\]
showing that $T$ is $(U,S)$-reduced.
\end{proof}

Here is a more concrete family of examples to illustrate the concept.

\begin{example}
Let $\mc{T}$ be a locally finite regular tree, let $G = \Aut(\mc{T})$ and for each vertex $v \in V\mc{T}$ and $k \ge 0$, let $G_{v,k}$ be the pointwise fixator of the ball of radius $k$ around $v$.  Let $H$ be a closed subgroup of $G$.  The \defbold{$\propP{k}$-closure} $H^{\propP{k}}$ of $H$ consists of all $g \in G$ such that, for each $v \in V\mc{T}$, there is some $h \in H$ with the same action as $g$ on the $k$-ball around $v$, that is, such that $gG_{v,k} = hG_{v,k}$.  The $\propP{0}$-closure is just the largest subgroup of $G$ that stabilizes setwise each $H$-orbit on vertices; for $k \ge 1$ then $H^{\propP{k}}$ is the smallest closed subgroup of $G$ containing $H$ that has property $\propP{k}$.  (See \cite{BanksElderWillis} for more information about these properties.)  One sees that we have a descending chain
\[
H^{\propP{0}} \ge H^{\propP{1}} \ge \dots
\]
of closed subgroups of $G$, with intersection $H$.

Now choose a set $A$ of representatives of the $H$-orbits on $V\mc{T}$ and fix $k \ge 0$.  For all vertices $v$ of $\mc{T}$ it is clear that $H^{\propP{k}}$ is a subgroup containing $H$ and contained in $HG_{v,k}$, so $H^{\propP{k}} \le T_{G_{v,k}}(H)$.  On the other hand, let $L = \bigcap_{v \in A} T_{G_{v,k}}(H)$ and consider $g \in L$ and $v \in V\mc{T}$.  Then there is some $h \in H$ such that $hv \in A$; from the definition of $L$, we see that $hgh\inv \in HG_{hv,k}$.  After conjugating by $h\inv$ we have $g \in HG_{v,k}$, showing that $g \in H^{\propP{k}}$.  Thus
\[
H^{\propP{k}}  = \bigcap_{v \in A} T_{G_{v,k}}(H).
\]
In particular, if $H$ is vertex-transitive, then $H$ is $G_{v,k}$-thick if and only if $H$ has property $\propP{k}$.
\end{example}

We note that $U$-thickness implies a certain dynamical property on the coset space.

\begin{defn}
Let $G$ be a group acting on a uniform space $(X,\ms{U})$.  We say the action is \defbold{expansive} at $x \in X$ there is an entourage $E_x \in \ms{U}$ such that if $(gx,gy) \in E_x$ for all $g \in G$, then $x=y$.  The action is \defbold{expansive} on $X$ if such an entourage $E_x$ exists for all $x \in X$, and \defbold{uniformly expansive} if a single entourage $E = E_x$ can be chosen for all $x \in X$.

Now suppose that $G$ is a Hausdorff topological group and let $H$ be a closed subgroup of $G$.  We equip $G/H$ with the right uniformity: basic entourages in $G/H$ are sets of the form
\[
E_U = \{ (xH,yH) \in G/H \times G/H \mid xH \subseteq UyH \},
\]
where $U$ is a symmetric identity neighbourhood in $G$.
\end{defn}

\begin{prop}\label{prop:coexpansive}
Let $G$ be a Hausdorff topological group and let $H$ be a closed subgroup of $G$.
\begin{enumerate}[(i)]
\item Suppose that $H$ is $U$-thick for some open subgroup $U$ such that $UH/H$ is compact.  Then $G$ is uniformly expansive on $G/H$.
\item Suppose that $G$ is a \tdlc group, $G$ is expansive at the trivial coset on $G/H$ and $G/H$ is compact.  Then $H$ is $U$-thick for some compact open subgroup $U$ of $G$.
\end{enumerate}
\end{prop}

\begin{proof}
(i)
Given Proposition~\ref{prop:thick}(i), we have
\[
H = \bigcap_{h \in H}hUH.
\]
Consider now $y \in G \setminus H$.  Then there exists $h \in H$ such that $y \not\in hUH$; equivalently, $Uh\inv yH \neq Uh\inv H$.  Thus $(h\inv H,h\inv yH) \not\in E_U$.  Thus $E_U$ is an entourage of expansivity for the action of $H$ on $G/H$ at the trivial coset.

Now consider an arbitrary pair of distinct points $xH,yH \in G/H$.  Then we can write $y = xz$ where $z \in G \setminus H$, and there is $h \in H$ such that $(hH,hzH) \not\in E_U$.  We can rewrite $(hH,hzH)$ as
\[
(hH,hzH) = ((hx\inv)xH,(hx\inv)yH);
\]
since $hx\inv \in G$, this shows that $E_U$ is an entourage of expansivity for the action of $G$ on $G/H$ at all points, so the action is uniformly expansive.

(ii)
We suppose that $E_V$ is an entourage of expansivity for the action of $G$ on $G/H$ at the trivial coset, where $V$ is a compact open subgroup of $G$.  Since $G/H$ is compact, there are only finitely many $(V,H)$-double cosets $Vg_1H,\dots,Vg_nH$ in $G$.  Given $y \in G \setminus H$ there is some $g_y \in G$ such that $Vg_yH \neq Vg_yyH$.  We are free to replace $g_y$ with any element of the coset $Vg_y$, so we may take $g_y = g_{i(y)}h_y$ for some $h_y \in H$ and $1 \le i \le n$.  We now have 
\[
UH = Uh_yH \neq Uh_yyH,
\]
where $U = \bigcap^n_{i=1}g\inv_i Vg_i$; note that $U$ is a compact open subgroup of $G$ that does not depend on $y$.  In particular, $h_yy \not\in UH$, so $y \not\in h\inv_y UH$.  Since $y \in G \setminus H$ was arbitrary, we conclude that
\[
H = \bigcap_{h \in H}hUH,
\]
so $H$ is $U$-thick.
\end{proof}

Proposition~\ref{prop:coexpansive}(ii) fails if we do not assume $G/H$ is compact.  Consider for example the case $H=\triv$: the left translation action of $G$ on itself is expansive if and only if there is an identity neighbourhood $U$ such that $\bigcap_{g \in G}gUg\inv = \triv$.  (Conjugation arises here because we are considering \emph{left} translation with respect to the \emph{right} uniformity.)  So if $G$ is a nondiscrete topologically simple group, then $G$ acts expansively on $G/\triv$, but clearly $\triv$ is not $U$-thick for any open subgroup $U$ of $G$.

\subsection{Thin factors}

We now introduce sharper versions of types $\tC$ and $\tD$.

\begin{defn}
Let $G$ be a Hausdorff topological group, let $H$ be a closed subgroup of $G$ and let $O$ and $S$ be subsets of $G$ such that $OH/H$ is a compact neighbourhood of the identity in $G/H$.  With respect to $(O,S)$, we say that a $G$-factor $\pi: G/H \rightarrow G/K$ (or $K/H$) is
\begin{itemize}
\item Type $\tC_O$ if $K \subseteq OH$.
\item Type $\tD_O$ if $K \cap O = H \cap O$.
\item Type $\tC_{O,S}$ if $O$ is a subgroup of $G$, $K = T_O(H)$ and $H$ is $(O,S)$-reduced.
\item Type $\tD_{O,S}$ if $K \cap O = H \cap O$ and $H$ is $(O,S)$-reduced.
\item Type $\tI_{O,S}$ if $H$ is maximal among proper closed subgroups of $K$ and $K$ is $(O,S)$-reduced.
\end{itemize}
A factorization of a $G$-space is a \defbold{$\tCDI_{O}$ factorization} if each $G$-factor in the sequence is of type $\tC_O$, $\tD_O$ or $\tI$ for some fixed $O$, and a \defbold{$\tCDI_{O,S}$ factorization} if each $G$-factor in the sequence is of type $\tC_{O,S}$, $\tD_{O,S}$ or $\tI_{O,S}$, for some fixed $O$ and $S$.
\end{defn}

The following observations are all easily verified.

\begin{lem}
Let $G/H$ be a coset space, let $O$ be a subset of $G$ such that $OH/H$ is a compact neighbourhood of the identity $G/H$, and let $S$ be a symmetric subset $S$ of $H$ containing the identity.  Let $K$ be a closed subgroup of $G$ containing $H$.
\begin{enumerate}[(i)]
\item If $K/H$ is of type $\tC_O$, then it is of type $\tC$.
\item If $K/H$ is of type $\tD_O$, then it is of type $\tD$.
\item If $K/H$ is of types $\tC_O$ and $\tD_O$, then $K=H$.
\end{enumerate}
\end{lem}

\begin{defn}\label{defn:thin}
Let $G/H$ be a locally compact coset space, let $U \le G$ and let $K$ be a closed subgroup of $G$ containing $H$.  Say that $K$ is \defbold{$U$-thin over $H$} (or that $K/H$ or the $G$-factor $G/H \rightarrow G/K$ is \defbold{$U$-thin}) if $(K \cap U)H$ is a closed subgroup of $G$.
\end{defn}

In particular, if $U$ is a subgroup of $G$ such that $UH/H$ is compact and open in $G/H$, and $K/H$ is $U$-thin, then $K/H$ factorizes into a factor $(K \cap U)H/H$ of type $\tC_U$ and a factor $K/(K \cap U)H$ of type $\tD_U$.  Conversely, if $K/H$ decomposes into a factor $L/H$ of type $\tC_U$ and a factor $K/L$ of type $\tD_U$, we see immediately that $L$ must be $(K \cap U)H$, so $K/H$ is $U$-thin.

The motivation for Definition~\ref{defn:thin} is that we have a criterion for $K/H$ to be $U$-thin, given in the next lemma, that can be leveraged (along with compactness) to obtain chain conditions in $\Sub(G/H)$.

\begin{lem}\label{lem:thin_criterion}
Let $G$ be a topological group, with closed subgroups $H$ and $U$ such that $UH$ is a neighbourhood of the identity in $G$, and a subset $S$ of $H$ such that $H = \grp{S, \N_H(U)}$.  Then for each closed subgroup $K$ containing $H$, $K/H$ is $U$-thin if and only if $K \cap sUH \subseteq UH$ for all $s \in S \cup S\inv$.  Moreover, if $K/H$ is $U$-thin and $UH/H$ is compact, then 
\[
K \cap HUH = (K \cap U)H \le T_U(H).
\]
\end{lem}

\begin{proof}
Let $P = (K \cap U)H$; note that, since $H \le K$, we have $P = K \cap UH$.  In particular, $P$ is a neighbourhood of the identity in $K$.  For the purposes of this proof we may replace $S$ with $S \cup S\inv$ and assume that $S$ is symmetric.

Suppose $P$ is a closed subgroup and let $g \in K \cap sUH$ for $s \in S$.  Then $g = suh$ for some $u \in U$ and $h \in H$.  Since $s,g,h \in K$, we in fact have $u \in K \cap U$.  Hence
\[
g \in S(K \cap U)H \subseteq SP = P \subseteq UH,
\]
showing that $K \cap sUH \subseteq UH$ for all $s \in S$.

Conversely, suppose that $K \cap sUH \subseteq UH$ for all $s \in S$ and let $L = \grp{P}$.  Then certainly $P \subseteq L$, so $L$ is a subgroup containing $H$; moreover, since $L$ contains an identity neighbourhood in $K$, it is open in $K$ and hence closed in $G$.  To show that $L = P$, it suffices to show that $lP=P$ for all $l \in L$.  We note that
\[
L = \grp{K \cap U, H} = \grp{K \cap U, \N_H(U), S};
\]
thus it is enough to show that $P$ is $K \cap U$-invariant, $\N_H(U)$-invariant and $S$-invariant under left multiplication.  That $P$ is $K \cap U$-invariant is clear; we also see, given $h \in \N_H(U)$, that $h$ normalizes $K \cap U$ and hence
\[
hP = h(K \cap U)H = (K \cap U)hH = (K \cap U)H = P.
\]
Consider now $s \in S$: since $s \in K$, and using our assumption that $K \cap sUH \subseteq UH$, we have
\[
sP = s(K \cap UH) = K \cap sUH \subseteq K \cap UH = P.
\]
By applying the same argument to $s\inv$ (recalling that $S$ is symmetric), in fact we have $sP = P$.  Thus $P = L$.

Suppose now that $K/H$ is $U$-thin.  Then $L = P$ is a subgroup of $G$ such that $H \le L \subseteq UH$, so $L \le T_U(H)$.  Moreover, for all $h \in H$ we have
\[
K \cap hUH = h(K \cap UH) = hL = L,
\]
showing that $K \cap HUH = L$.
\end{proof}

In particular, under mild assumptions, $U$-thin subgroups over $H$ satisfy an ascending chain condition.

\begin{cor}\label{cor:ascending}
Let $G$ be a topological group, with closed subgroups $H$ and $U$ such that $UH$ is open in $G$ and $UH/H$ is compact, and a subset $S$ of $H$ such that $H = \grp{S, \N_H(U)}$.  Let $\{K_i \mid i \in I\}$ be a family of closed $U$-thin subgroups over $H$ directed by inclusion and let $K = \ol{\bigcup_{i \in I}K_i}$.  Then $K$ is $U$-thin over $H$.
\end{cor}

\begin{proof}
By Lemma~\ref{lem:thin_criterion}, for each $i \in I$ we have $K_i \cap HUH \subseteq T_U(H)$.  Since $UH$ is open, so is $HUH$, and hence every point in $K \cap HUH$ is approximated by points in $K' \cap HUH$, where $K' = \bigcup_{i \in I}K_i$.  Since $T_U(H)$ is closed we deduce that $K \cap HUH \subseteq T_U(H)$, and hence $K/H$ is $U$-thin.
\end{proof}

\begin{cor}\label{cor:thin_reduction}
Let $G$ be a topological group, with closed subgroups $H$ and $U$ such that $UH$ is open in $G$ and $UH/H$ is compact.  Then $K$ is $U$-thin over $H$ if and only if $\grp{H,K \cap U}$ is $U$-thin over $H$.
\end{cor}

\begin{proof}
Let $K' = \grp{H,K \cap U}$; we use throughout the characterization of the $U$-thin property given by Lemma~\ref{lem:thin_criterion}, taking $S=H$.  Clearly, if $K \cap HUH \subseteq UH$, then also $K' \cap HUH \subseteq UH$ so if $K$ is $U$-thin then $K'$ is $U$-thin.  On the other hand, if $K$ is not $U$-thin, then there exist $h \in H$ and $g \in K \cap hUH$ such that $g \not\in UH$, say $g = huh'$, for $h,h' \in H$ and $u \in U$; we see then that $u = h\inv g(h')\inv$ is an element of $K \cap U$.  Thus $g \in H(K \cap U)H \subseteq K'$, showing that $K'$ is not $U$-thin.
\end{proof}

\begin{rmk}
Suppose that $G$ is a \tdlc group, $S$ is a compact symmetric subset of $G$, $U$ is a compact open subgroup of $G$ and $H = \grp{S}$.  Then the closed subgroups $K$ of $G$ that are $U$-thin over $H$ form a subspace $\Sub(G/H)_U$ of $\Sub(G/H)$.  In this situation, Lemma~\ref{lem:thin_criterion} is equivalent to saying that for $K \in \Sub(G/H)$, we have $K \in \Sub(G/H)_U$ if and only if $K \cap SU \subseteq UH$, and we also note that $H \cap SU \subseteq UK$.  Thus $\Sub(G/H)_U$ is a clopen subspace of $\Sub(G/H)$.  Clearly also $\Sub(G/H)_U$ is a down-set in $\Sub(G/H)$ and its complement in $\Sub(G/H)$ is an up-set in $\Sub(G)$.  We will obtain an irreducible factor in the $\tCDI$ factorization of $G/H$ by pairing a minimal element $K$ of $\Sub(G/H) \setminus \Sub(G/H)_U$ with a maximal element of $\Sub(G/H)_U \cap \Sub(K)$.
\end{rmk}

\subsection{The degree and further chain conditions}

To obtain a descending chain condition complementary to Corollary~\ref{cor:ascending}, we need an additional finiteness condition.

\begin{lem}\label{lem:descending_chain}
Let $G$ be a topological group, with a family $\{K_i \mid i \in I\}$ of closed subgroups directed under reverse inclusion; let $U$ be a closed subgroup of $G$ and $S$ a subset of $K = \bigcap_{i \in I}K$.  Let $H$ be a closed subgroup of $K$ such that $H = \grp{S,\N_H(U)}$ and $K/H$ is $U$-thin (for example, $H = \grp{S,\N_K(U)}$ will work).  Suppose also that $UH$ is open and that $\bigcup_{s \in S \cup S\inv}sUH/H$ is compact.  Then there is $i_0 \in I$ such that $K_i/H$ is $U$-thin for all $i \ge i_0$.
\end{lem}

\begin{proof}
Our hypotheses ensure that $Y/H$ is a compact subset of $G/H$, where
\[
Y = \bigcup_{s \in S \cup S\inv}sUH \setminus UH.
\]

By Lemma~\ref{lem:thin_criterion}, the intersection $K \cap Y$ is empty; moreover, both $Y$ and each of the groups $K_i$ is a union of left cosets of $H$.  Since $Y/H$ is compact, and since the family $\{K_i \mid i \in I\}$ is directed by reverse inclusion, there must be some $i_0 \in I$ such that $K_{i_0} \cap Y = \emptyset$, and hence $K_i \cap Y = \emptyset$ for all $i \ge i_0$.  By Lemma~\ref{lem:thin_criterion}, $K_i/H$ is $U$-thin over $H$ for all $i \ge i_0$.
\end{proof}

For the purposes of an induction argument later, we introduce a certain numerical invariant.

\begin{defn}\label{defn:degree}
Let $G$ be a topological group, with closed subgroups $H$ and $U$ and a subset $S$ of $H$.  Then the \defbold{degree} $\deg(G/H,U,S)$ of $(G/H,U,S)$ is the smallest cardinality of a subset $Y$ of $G$ containing the identity such that $\bigcup_{s \in S \cup S\inv}sUH \subseteq \bigcup_{y \in Y}UyH$.
\end{defn}

\begin{lem}\label{lem:degree_control}
Let $G$ be a topological group, with closed subgroups $H$ and $U$ and a subset $S$ of $H$.  Suppose that $\bigcup_{s \in S \cup S\inv}sUH/H$ is compact, that $UgH/H$ is compact and open in $G/H$ for all $g \in G$ and that $H = \grp{S,\N_H(U)}$.
\begin{enumerate}[(i)]
\item The degree of $(G/H,U,S)$ is finite.
\item Suppose that $K$ is a closed subgroup of $G$ containing $H$.  Then 
\[
\deg(G/K,U,S) \le \deg(G/H,U,S).
\]
If equality occurs, then $K$ is $U$-thin over $H$.
\end{enumerate}
\end{lem}

\begin{proof}
We may assume that $S$ is a symmetric subset containing the identity.  Let $O = \bigcup_{s \in S}sUH$.

(i)
By hypothesis $O/H$ is compact.  Since $U$ has open orbits on $G/H$, it follows that $O$ is contained in the union of finitely many $U$-orbits.

(ii)
Note that the hypotheses of the lemma still apply to $(G/K,U,S)$.  In particular, the $U$-orbits form a clopen partition of $G/K$, corresponding to the partition of $G$ into $(U,K)$-double cosets.

In light of (i), there are $y_1,\dots,y_n \in G$, with $n = \deg(G/H,U,S)$, such that $O \subseteq \bigcup^n_{i=1}Uy_iH$ and such that the double cosets $Uy_1H, \dots, Uy_nH$ are disjoint.  Since $1 \in O$, one of the double cosets is the trivial double coset, without loss of generality $y_1=1$.  We then have 
\[
\bigcup_{s \in S} sUK = OK \subseteq  \bigcup^n_{i=1}Uy_iK,
\]
so $\deg(G/K,U,S) \le \deg(G/H,U,S)$.  We see that equality occurs if and only if the double cosets $Uy_1K, \dots, Uy_nK$ are disjoint.  In particular, if this occurs then $UK \cap O = UH$, so
\[
\forall s \in S: UK \cap sUH \subseteq UH,
\]
and hence $K/H$ is $U$-thin by Lemma~\ref{lem:thin_criterion}.
\end{proof}

The degree also has an interesting consequence for ascending chains of closed subgroups.

\begin{lem}\label{lem:ascending_degree}
Let $G,H,U,S$ be as in Lemma~\ref{lem:degree_control}.  Let $\{K_i \mid i \in I\}$ be a family of closed overgroups of $H$, directed under inclusion, let $K = \ol{\bigcup_{i \in I}K_i}$ and suppose $V = K \cap U$ is open in $K$.  Then there is $i_0 \in I$ such that $K/\grp{H,K_i \cap V}$ is $U$-thin for all $i \ge i_0$.
\end{lem}

\begin{proof}
For each $i \in I$ let $L_i = \grp{H,K_i \cap V}$ and let $L = \ol{\bigcup_{i \in I}L_i}$.  Since $L$ is closed and contains $K_i \cap V$ for all $i \in I$, we have $V \le L$, and also $H \le L$, so $\grp{H,V} \le L$; conversely, it is clear that $\grp{H,V}$ is open in $K$, hence closed, and contains $L_i$ for all $i \in I$, so $L = \grp{H,V}$.

By Lemma~\ref{lem:degree_control}, $\deg(G/H,U,S)$ is finite and 
\[
\deg(G/L_i,U,S) \le \deg(G/L_j,U,S) \le \deg(G/H,U,S) \text{ for all } i \ge j.
\]
There is therefore some $i_0 \in I$ such that $\deg(G/L_{i_0},U,S)$ achieves its minimum value over all $i \in I$, and then for all $i \ge i_0$ we have $\deg(G/L_{i_0},U,S) = \deg(G/L_{i},U,S)$.  Thus, after replacing $I$ with the set $\{i \in I \mid i \ge i_0\}$, we may assume $\deg(G/L_i,U,S)$ is constant over $i \in I$.  It then follows that for all pairs $i,j \in I$ such that $i \le j$, the coset space $L_j/L_i$ is $U$-thin.  By Corollary~\ref{cor:ascending}, the coset space $L/L_i$ is also $U$-thin for all $i \in I$.  By Corollary~\ref{cor:thin_reduction}, we deduce that $K/L_i$ is $U$-thin for all $i \in I$.
\end{proof}

In particular, we deduce the following about containment in maximal subgroups of a \tdlc group.

\begin{prop}\label{prop:cdi_maximal_in_G}
Let $G$ be a \tdlc group, let $H$ be a compactly generated closed subgroup of $G$ and let $U$ be a compact open subgroup of $G$.  Suppose that $G = \grp{H,U}$.  Then at least one of the following holds:
\begin{enumerate}[(i)]
\item There is an ascending chain $(H_i)$ of proper closed subgroups of $G$, each containing $H$, such that $G = \overline{\bigcup H_i}$ and such that $H_i = \grp{H,H_i \cap U}$ and $G = UH_i$ for all $i \in I$;
\item $H$ is contained in a maximal proper closed subgroup $K$ of $G$.
\end{enumerate}
\end{prop}

\begin{proof}
We will assume that (i) is false, and prove (ii).

Let $S$ be a compact generating set for $H$.  Then it is clear that $G,H,U,S$ are as in Lemma~\ref{lem:degree_control}.

Let $\mc{H}$ be the set of closed subgroups $K$ such that $H \le K < G$ and $K$ is $(U,H)$-reduced.  Clearly $H \in \mc{H}$, so $\mc{H}$ is nonempty.  Let $H_+ = \overline{\bigcup_{i \in I} H_i}$, where $(H_i)_{i \in I}$ is an ascending chain of elements of $\mc{H}$.  Then 
\[
\forall i \in I: H_i = \grp{H,H_i \cap U} \le \grp{H,H_+ \cap U},
\]
and hence $H_+$ is $(U,H)$-reduced.  Suppose $G = H_+$.  Then by Lemma~\ref{lem:ascending_degree}, there is $i \in I$ such that $G/H_j$ is $U$-thin, in other words $UH_j$ is an open subgroup of $G$, for all $j \ge i$; since $G = \grp{H,U}$, it follows that $G = UH_j$ for all $j \ge i$.  Since (i) is assumed false, this is a contradiction.  So in fact $H_+ \in \mc{H}$.

We now apply Zorn's lemma to obtain a maximal element $K$ of $\mc{H}$.  By the definition of $\mc{H}$, we see that $K$ is closed, proper, and $(U,H)$-reduced.  Since $\grp{H,U} = G$ we must have $K \cap U < U$.  Let $\mc{K}$ be the set of proper closed subgroups of $G$ containing $K$.  Then given $L \in \mc{K}$, we have $\grp{H,L \cap U} = K$ by the maximality of $K$; in particular, $L \cap U = K \cap U < U$.  Taking an ascending chain $(K_i)_{i \in I}$ in $\mc{K}$, with $K_+ := \ol{\bigcup_{i \in I}K_i}$, we therefore find that $K_+ \cap U = K \cap U$, so $K_+ < G$ and hence $K_+ \in \mc{K}$.  Applying Zorn's lemma again, we obtain a maximal element $L$ of $\mc{K}$, which is then a maximal proper closed subgroup of $G$ that contains $H$.
\end{proof}

\subsection{The $\tCDI$ factorization for zero-dimensional coset spaces}

In this subsection, we will obtain a $\tCDI$ factorization for a certain class of zero-dimensional coset spaces.

\begin{defn}
Let $G/H$ be a locally compact zero-dimensional coset space.  We say $G/H$ has \defbold{type CG} (with respect to a subgroup $U$ of $G$ and a subset $S$ of $H$) if the following conditions are satisfied:
\begin{enumerate}[(a)]
\item $UgH/H$ is compact and open in $G/H$ for all $g \in G$;
\item $\bigcup_{s \in S \cup S\inv} sUH/H$ is compact;
\item $H$ is $(U,S)$-reduced.
\end{enumerate}
\end{defn}

If $G/H$ is of type CG and $H \le K \le G$, then the coset space $G/K$ satisfies conditions (a) and (b) of type CG with respect to $U$ and $S$.  Thus $G/K$ is itself of type CG if and only if $K$ is $(U,S)$-reduced, or equivalently, if and only if $K = \grp{H,K \cap U}$.

Given a \tdlc group $G$, a compact subset $S$ such that $H = \grp{S}$ is closed and a compact open subgroup $U$ of $G$, it is clear that $G/H$ is of type CG with respect to $U$ and $S$.  We can now state and prove a factorization result for coset spaces of type CG.

\begin{prop}\label{prop:cdi_disconnected}
Let $G/H$ be a coset space of type CG with respect to a subgroup $U$ of $G$ and a subset $S$ of $H$.  Then $G/H$ admits a $\tCDI_{U,S}$ factorization.
\end{prop}

\begin{proof}
We proceed by induction on $\deg(G/H,U,S)$.  Without loss of generality, assume $S = S\inv$ and $1 \in S$.

As a first step, we can pass from $G/H$ to $G/T_U(H)$.  The factor $T_U(H)/H$ is then of type $\tC_{U,S}$; $T_U(H)$ is $(U,S)$-reduced; and $\deg(G/T_U(H),U,S) \le \deg(G/H,U,S)$.  So without loss of generality, we can assume that $H$ is $U$-thick.  We may then assume that $U \nleq H$, in other words, that $G/H$ is not of type $\tD_{U,S}$.  In that case, since $H$ is $U$-thick, we see that $UH$ is not a subgroup and hence $G$ is not $U$-thin over $H$.

Let $\mc{K}_1$ be the set of all closed overgroups of $H$ that are not $U$-thin over $H$.   By Lemma~\ref{lem:descending_chain} and Zorn's lemma,  $\mc{K}_1$ has a minimal element $K_1$.  By Corollary~\ref{cor:thin_reduction} we see that $K_1$ is $(U,S)$-reduced.  Now let $\mc{K}_2$ be the set of closed subgroups $H \le K \le K_1$ such that $K/H$ is $U$-thin.  For each $K \in \mc{K}_2$ we have a closed subgroup $(K \cap U)H \subseteq UH$; since $H$ is already $U$-thick, this is only possible if $(K \cap U)H = H$, in other words, $K \cap U = H \cap U$.  By Corollary~\ref{cor:ascending}, $\mc{K}_2$ has a maximal element $K_2$, and since $K_2/H$ is $U$-thin we have $K_2 \neq K_1$.  On the other hand, we see that no closed subgroup $K$ of $G$ can lie strictly between $K_2$ and $K_1$: if $K$ were $U$-thin it would contradict the maximality of $K_2$, whereas if it were not $U$-thin it would contradict the minimality of $K_1$.

The properties we have shown so far demonstrate that $K_2/H$ is of type $\tD_{U,S}$, while $K_1/K_2$ is of type $\tI_{U,S}$.  Since $K_1$ is $(U,S)$-reduced, the coset space $G/K_1$ also satisfies the hypotheses of the proposition and we can continue the factorization.  Moreover, since $K_1/H$ is not $U$-thin, by Lemma~\ref{lem:degree_control} we have 
\[
\deg(G/K_1,U,S) < \deg(G/H,U,S).
\]
We now conclude that $G/H$ has a $\tCDI_{U,S}$-factorization, by the induction hypothesis.
\end{proof}

We can extract another variant of Proposition~\ref{prop:cdi_maximal_in_G} from the proof of Proposition~\ref{prop:cdi_disconnected}.

\begin{cor}\label{cor:cdi_maximal}
Let $G/H$ be a coset space of type CG with respect to a subgroup $U$ of $G$ and some $S \subseteq H$; for example, $G$ is a \tdlc group, $H$ is a compactly generated closed subgroup and $U$ is a compact open subgroup.  Suppose that $HU$ is not a subgroup of $G$.  Then there is a closed subgroup $V$ of $U$ such that $H$ is contained in a maximal proper closed subgroup of $K := \grp{H,V}$ and $K \cap U = V$.
\end{cor}

\begin{proof}
There is no loss of generality in replacing $H$ with $T_U(H)$, so we may assume $H$ is $U$-thick.  Obtain the subgroups $K_1$ and $K_2$ as in the proof of Proposition~\ref{prop:cdi_disconnected}.  Then $K_2$ is a maximal proper closed subgroup of $K_1$.  Moreover, $K_1$ is $(U,S)$-reduced and contains $H$, so $K_1 = \grp{S,K_1 \cap U} = \grp{H,K_1 \cap U}$; setting $V = K_1 \cap U$ yields the desired conclusion.
\end{proof}

\begin{rmk}
At this point it is worth considering another phenomenon in \tdlc groups related to coset decompositions.  In \cite[Theorem~B]{ReidDistal}, it was shown that given a \tdlc group $G$ and a compactly generated closed subgroup $H$, there is an open subgroup $E$ of $G$, called a \defbold{reduced envelope} for $H$, with the properties that $H \le E$ and for all open subgroups $O$ of $G$, one has
\[
|H:H \cap O| < \infty \Leftrightarrow |E:E \cap O| < \infty.
\]
Thus up to finite index, $E$ is the unique smallest open subgroup of $G$ containing $H$.  Moreover, for a sufficiently small compact open subgroup $U$, one can choose $E$ to be of the form $E = HRU$ where $R := \Res_G(H)$ is the intersection of all open $H$-invariant subgroups of $G$, and one also has $R = \Res_E(E)$, so $R$ is normal in $E$.  It is also the case that $R$ has no proper $H$-invariant open subgroup; indeed, given a proper closed $H$-invariant subgroup $S$ of $R$, then the conjugation action of $H$ on $R/S$ is nondistal (\cite[Proposition~3.6]{ReidDistal}; a \defbold{nondistal} action of a group $H$ on a Hausdorff space $X$ means there exist distinct points $x,y \in X$ and a net $(h_i)$ in $H$ such that $h_ix$ and $h_iy$ converge to the same point).  One thus obtains the following sequence of closed subgroups:
\[
H \le L \le E \le G,
\]
where $L = \ol{HR}$; one can also characterize $L$ as the intersection of all open subgroups of $G$ containing $H$.  We see that $E/L$ is of type $\tC_U$, while $G/E$ is of type $\tD_U$.  The factor $L/H$, assuming it is nontrivial, then admits some $\tCDI_{U}$ factorization; what we know from the way $L$ was constructed is that no proper open subgroup of $L$ contains $H$, so the uppermost factor $L/M$ in the $\tCDI_{U}$ factorization of $L/H$ must be of type $\tC_U$ (that is, $L = (L \cap U)M$) or type $\tI$.  The structure of $L/M$ is further restricted by the fact that one has a continuous $H$-equivariant map with dense image
\[
\theta: R/(M \cap R) \rightarrow L/M; \; g(M \cap R) \mapsto gM,
\]
where $H$ acts nondistally on $R/(M \cap R)$ by conjugation, and hence also $H$ acts nondistally on $L/M$ by translation.
\end{rmk}

\subsection{The $\tCDI$ factorization for locally compact coset spaces}\label{sec:connected}

We can now combine the results of the previous subsection with some basic properties of Lie groups to obtain a $\tCDI$ factorization theorem for coset spaces of locally compact groups, as stated in the introduction.

\begin{proof}[Proof of Theorem~\ref{thm:cdi}]
Without loss of generality we can replace $G$ with $G/R$ and $H$ with $HR/R$, so we may assume that $G^\circ$ is a Lie group.

By Lemma~\ref{lem:vD} there is an open subgroup $U$ of $G$ such that $U \subseteq OG^\circ$ and $U/G^\circ$ is compact.  Fix a compact generating set $S$ for $H$ such that $1 \in S$ and $S = S\inv$.

Suppose that $H^\circ < G^\circ$, and let $\mc{K}$ be the set of closed subgroups $H \le K \le G$ such that $H^\circ < K^\circ$.  Since $G^\circ$ is finite-dimensional, it is easy to see that $\mc{K}$ has a minimal element $K$.  The minimality of $K$ ensures that $K/H$ is connected, so $K = \grp{H,K \cap O}$.  There is then an almost connected open subgroup $V$ of $K$ contained in $U \cap K$; note that $K = VH$.  Let $\mc{L}$ be the set of closed subgroups $H \le L < K$.  Then given $L \in \mc{L}$, we have $L^\circ = H^\circ$, so the coset space $L/H$ is zero-dimensional, and we see (after dividing out by $H^\circ$) that in fact $L/H$ is a coset space of type CG with respect to $W$ and $S$, where $W$ is an almost connected open subgroup of $L$ contained in $V$.  By Proposition~\ref{prop:cdi_disconnected}, $L/H$ admits a $\tCDI_{W,S}$ factorization.  In particular, all the factor spaces in this factorization are as in (i)--(iii) of the theorem.  We can choose a suitable $L \in \mc{L}$ in one of two ways.  If $\mc{L}$ has a maximal element, take $L \in \mc{L}$ maximal; then the factor $K/L$ is as in (iii).  On the other hand if $\mc{L}$ does not have a maximal element, then there is an increasing sequence $(L_i)_{i \in I}$ in $\mc{L}$ such that $\bigcup_{i \in I} L_i$ is dense in $K$, so $\bigcup_{i \in I}L_i \cap V$ is dense in $V$.  It then follows from Lemma~\ref{lem:comp_gen:approximation} that $V = (O \cap V)L_j$ for some $j \in I$.  We now take $L = L_j$, so that
\[
K = VH = (O \cap V)L.
\]
In this case, we see that $L$ is cocompact in $K$, so it is compactly generated, and the factor $K/L$ is as in (ii).  In either case, having factorized $K/H$ in the required way, we can continue the argument with the coset space $G/K$ in place of $G/H$.

By repeating the argument of the last paragraph and using induction on $\dim_{\Rb}(G^\circ) - \dim_{\Rb}(H^{\circ})$, we reduce to the case that $H^\circ = G^\circ$.  In this case, we see that $G/H$ is of type CG with respect to $U$ and $S$, so by Proposition~\ref{prop:cdi_disconnected} it admits a $\tCDI_{U,S}$ factorization.  As before, all the factor spaces in this factorization are as in (i)--(iii) of the theorem, and we are done.
\end{proof}

\subsection{Example: Groups of tree automorphisms fixing one end}

The following is a natural class of examples of transitive actions of \tdlc groups on noncompact locally compact spaces where the $\tCDI$ factorization may be useful.

\begin{defn}
Let $\mc{T}$ be a locally finite leafless tree.  A closed subgroup $G$ of $\Aut(\mc{T})$ is \defbold{almost boundary-transitive} if $G$ fixes one end $\xi$ of $\mc{T}$ and acts transitively on the remaining ends $\partial \mc{T} \setminus \{\xi\}$.
\end{defn}

The transitive action we wish to consider here is the action of $G$ on the locally compact (usually not compact) space $X = \partial \mc{T} \setminus \{\xi\}$, where $\partial \mc{T}$ has its usual compact topology, on which $G$ acts continuously.  By Lemma~\ref{lem:transitive_topology}, we can identify $X$ with the coset space $G/G_{\xi'}$, where $\xi'$ is any end of $T$ other than $\xi'$.  There are a few possibilities.

\begin{lem}\label{lem:tree_end_types}
Let $\mc{T}$ be a locally finite leafless tree and let $G$ be a closed subgroup of $\Aut(\mc{T})$ fixing one end $\xi$ and acting transitively on the remaining ends.  Let $X = \partial \mc{T} \setminus \{\xi\}$, let $\xi' \in X$, let $\mc{L}$ be the line between $\xi$ and $\xi'$ and let $H = G_{\xi'}$.  Then $H$ is compactly generated and exactly one of the following holds.
\begin{enumerate}[(i)]
\item (Isolated end case) $\xi$ is an isolated point in $\partial \mc{T}$, $X$ is compact and $G$ acts without translation.
\item (Lineal case) $\mc{T} = \mc{L}$ and $G = H = \Zb$, acting by translation.
\item (Horocyclic case) $\xi$ is an accumulation point in $\partial \mc{T}$ but $G$ acts without translation; in particular, $H \le G_v$ for every $v \in V\mc{L}$, so $X$ has a $\tCDI$ factorization
\[
X \rightarrow G/G_v \rightarrow \{\ast\}.
\]
Moreover, every compactly generated subgroup of $G$ fixes a vertex of the tree, so $G$ is not compactly generated; also, $G$ has infinitely many orbits on $V\mc{T}$, since it preserves the horospheres around $\xi$.
\item (Focal case) $H$ has a semidirect decomposition $H = U_0 \rtimes \grp{s}$, where $U_0$ is the pointwise fixator of $\mc{L}$ and $s$ is a translation towards $\xi$ of length $l > 0$.  The space $\partial \mc{T}$ is homeomorphic to the Cantor set.  Moreover, $G$ has exactly $l$ orbits on $V\mc{T}$: each orbit is a union of horospheres, where two horospheres belong to the same orbit if and only if the distance between them is a multiple of $l$. In particular, $G$ is compactly generated.
\end{enumerate}
\end{lem}

\begin{proof}
Note that in all cases we have $X \cong G/H$ by Lemma~\ref{lem:transitive_topology}.

Let $U_0$ be the pointwise fixator of $\mc{L}$.  We see that $H$ acts on $\mc{L}$, so either $H = U_0$ or $H = U_0 \rtimes \grp{s}$ where $s$ is a translation towards $\xi$.  In either case, $U_0$ is compact, so $H$ is compactly generated.

The cases (i)--(iv) are clearly mutually exclusive.  Suppose $\xi$ is isolated; then $X$ is clearly compact.  If $G$ acts with translation, then $G$ has a translation with attracting end $\xi$; the fact that $\xi$ is isolated then means there are no ends other than $\xi$ and $\xi'$, and then since $\mc{T}$ is leafless, we must have $\mc{T} = \mc{L}$, and we see that $G = H = \Zb$.   We may assume from now on $\xi$ is an accumulation point in $\partial \mc{T}$.

Suppose $H = U_0$, that is, $H$ acts without translation.  Then $\mc{L}$ is not an axis of translation of $G$.  Since $G$ acts transitively on $X$, in fact no axis of translation of $G$ has $\xi$ as an end.  Since $G$ fixes $\xi$, we conclude that $G$ acts without translation, and hence each element fixes pointwise a ray representing $\xi$.  It is then easy to see, via Lemma~\ref{lem:comp_gen:approximation} and the fact that vertex stabilizers are open, that every compactly generated subgroup fixes pointwise a ray representing $\xi$.  The remaining assertions for the horocyclic case are also easily verified.

From now on we assume the remaining possibility for $H$, that is, $H = U_0 \rtimes \grp{s}$ where $s$ is a translation, of length $l$ say.  From the structure of $H$ we see that $s$ achieves the minimum translation length of $H$, and hence the minimum translation length of $G$ along the axis $\mc{L}$.  By considering how $s$ acts on the other ends, we see that $\xi'$ is an accumulation point in $\partial \mc{T}$.  Since $G$ acts transitively on $X$, in fact every point of $\partial \mc{T}$ is an accumulation point, and indeed every point in $X$ is the end of the axis of some $G$-conjugate of $s$, where the other end of the axis is $\xi$.  All axes of translation of $G$ are achieved in this way, so $G$ has minimum translation length $l$.  Since $\partial \mc{T}$ is perfect, second-countable, totally disconnected and compact, we see that $\partial \mc{T}$ is homeomorphic to the Cantor set.

We next observe that every $G$-orbit on $V\mc{T}$ intersects $\mc{L}$: given a vertex $v \in V\mc{T}$, then since $v$ is not a leaf, $v$ lies on some line $\mc{L}'$ between $\xi$ and some other end $\xi''$.  We see that in fact $\mc{L}' = g\mc{L}$ for some $g \in G$, so $g\inv v \in V\mc{L}$.  Since $\grp{s}$ has $l$ orbits on $V\mc{L}$, it follows that $G$ has at most $l$ orbits on $V\mc{L}$.  On the other hand, the horospheres around the fixed end $\xi$ form a system of imprimitivity for $G$, and since $G$ has minimum translation length $l$, two horospheres can only intersect the same $G$-orbit if their distance is a multiple of $l$.  Thus the orbits of $G$ on $V\mc{T}$ are as described in the focal case.  The fact that $G$ acts with finitely many orbits on $V\mc{T}$ ensures that $G$ is cocompact in $\Aut(\mc{T})$, and hence $G$ is compactly generated.
\end{proof}

From now on, we will confine our attention to the focal case.  The motivation for considering such groups comes from two sources of examples.

\begin{enumerate}[(1)]
\item A \defbold{scale group action} of a \tdlc group is a continuous, proper, vertex-transitive and almost boundary-transitive action on a locally finite tree $\mc{T}$ with more than two ends.  Scale groups were recently studied by G. Willis in \cite{WillisSG}, and play an important role in the general theory of \tdlc groups; in effect, whenever one has a \tdlc group $G$ and an automorphism $\alpha$ of $G$ that does not stabilize any compact open subgroup, then the large-scale dynamics of $\alpha$ can be represented by a scale group action (not necessarily unique) on a suitable tree.

\item Suppose we have a locally finite tree $\mc{T}$ and a closed subgroup $H$ of $\Aut(\mc{T}$) acting $2$-transitively on $\partial \mc{T}$.  As with $2$-transitive finite permutation groups, it is natural to ask what kind of group can appear as the point stabilizer $G = H_{\xi}$ for $\xi \in \partial \mc{T}$; clearly $G$ is then almost boundary-transitive.  Boundary-$2$-transitive subgroups of $\Aut(\mc{T})$ for a locally finite tree $\mc{T}$ are a rich source of compactly generated topologically simple \tdlc groups: it was shown by M. Burger and Sh. Mozes in \cite[Proposition~3.1.2]{BurgerMozes} that any such group has a cocompact topologically simple normal subgroup, and P.-E. Caprace and N. Radu showed in \cite{CapRad} that for a fixed $\mc{T}$, the topologically simple boundary-$2$-transitive closed subgroups of $\Aut(\mc{T})$ form a closed subspace of $\Sub(\Aut(\mc{T}))$.  Some of the groups that arise have been classified directly by Radu (\cite{Radu}).  Boundary-$2$-transitive actions on locally finite trees also play prominent roles in \cite{CCMT} and \cite{CaretteDreesen}.
\end{enumerate}

\begin{lem}\label{lem:tree_special_types}
Let $\mc{T}$ be a locally finite leafless tree and let $G$ be a closed subgroup of $\Aut(\mc{T})$.
\begin{enumerate}[(i)]
\item $G$ is a scale group on $\mc{T}$ if and only if the action is almost boundary-transitive and focal, with minimum translation length $1$.
\item Suppose that $\mc{T}$ has more than two ends, $G$ acts $2$-transitively on $\partial \mc{T}$ and let $\xi \in \partial \mc{T}$.  Then $G_{\xi}$ is almost boundary-transitive and focal, with minimum translation length at most $2$.
\end{enumerate}
\end{lem}

\begin{proof}
Part (i) is clear from Lemma~\ref{lem:tree_end_types}.

For (ii), by \cite[Lemma~3.1.1]{BurgerMozes}, the property of having $2$-transitive action on $\partial \mc{T}$ is equivalent to the property that every stabilizer of a vertex of $\mc{T}$ acts transitively on the boundary.  In particular, in this case we see that $G$ has $l \le 2$ orbits on vertices.  From there it is easy to see that $G_{\xi}$ is focal, with minimum translation length $l$.
\end{proof}

Note that in the focal case of Lemma~\ref{lem:tree_end_types}, $G/G_{\xi'}$ is neither compact nor discrete, but there is also no obvious $\tCDI$ factorization of $G/G_{\xi'}$ coming from the geometry of the tree.  The following example shows that, while the coset space $G/G_{\xi'}$ can be irreducible, in general the number of factors needed for a $\tCDI$ factorization depends on the action. 

\begin{example}
Let $M = \Qb^d_p$ and let $U = \Zb^d_p$ be a compact open subgroup of $M$.  Let $\mc{T}$ be the tree formed by the cosets of $p^nU$ in $M$, where two cosets are adjacent if one has index $p^d$ in the other.  Let $s$ act on $M$ by multiplication by $p\inv$ and form the semidirect product $G = M \rtimes \grp{s}$.  Then there is a natural scale group action of $G$ on $\mc{T}$, where $m.(n+p^kU) = m+n+p^kU$ for $m \in M$ and $s.(n+p^kU) = p\inv n + p^{k-1}U$.  Now take $\xi$ to be the end specified by the ray $(p^{-i}U)_{i \ge 0}$, and $\xi'$ the end specified by the ray $(p^iU)_{i \ge 0}$.  Then we see that $G$ fixes $\xi$, while $G_{\xi'} = \grp{s}$, so as a $G$-space, in this case $X \cong G/\grp{s}$.  We then have a $\tCDI$ factorization
\[
G/\grp{s} \rightarrow G/(M_1 \rtimes \grp{s}) \rightarrow \dots \rightarrow G/(M_d \rtimes \grp{s}) = \{\ast\}
\]
where $M_1 \subseteq M_2 \dots \subseteq M_d = M$ is a chain of $\grp{s}$-invariant subgroups of $M$, with $M_i \cong \Qb^i_p$.
\end{example}

In this context though we can make some general observations about the closed subgroups containing $G_{\xi'}$.

\begin{prop}\label{prop:focal_almost_transitive}
Let $\mc{T}$ be a locally finite leafless tree and let $G$ be a closed subgroup of $\Aut(\mc{T})$ with focal almost boundary-transitive action, with fixed end $\xi$ and minimum translation length $l$. Let $X = \partial \mc{T} \setminus \{\xi\}$, let $\xi' \in X$, let $\mc{L}$ be the line between $\xi$ and $\xi'$ and let $H = G_{\xi'}$.
\begin{enumerate}[(i)]
\item For each $K \in \Sub(G/H)$ there is a tree $\mc{T}_K$, the \defbold{minimal tree} of $K$, which is the smallest $K$-invariant subtree of $\mc{T}$.
\item Given $K,K' \in \Sub(G/H)$, we have $\mc{T}_K \subseteq \mc{T}_{K'}$ if and only if $K \le K'$.  In particular, $\mc{L} \subseteq \mc{T}_K$ for all $K \in \Sub(G/H)$, with equality only if $K = H$.
\item Given $K \in \Sub(G/H) \setminus \{H\}$, then $K$ has focal almost boundary-transitive action on $\mc{T}_K$ with fixed end $\xi$ and minimum translation length $l$.  In particular, $K$ is compactly generated, and if $G$ is a scale group on $\mc{T}$ then $K$ is a scale group on $\mc{T}_K$.
\item Let $U$ be a compact open subgroup of $G$ containing $U_0$ and let $K \in \Sub(G/H) \setminus \{H\}$.  Then $K$ is not is $U$-thin over $H$, and in addition, given $v \in V\mc{L}$, then $K_v$ does not fix $s\inv v$.
\item Every element of $\Sub(G/H) \setminus \{H\}$ contains a minimal such element.
\end{enumerate}
\end{prop}

\begin{proof}
Write $H = U_0 \rtimes \grp{s}$ as in Lemma~\ref{lem:tree_end_types}(iv).  In particular, $H$ is generated by the compact set $S = U_0 \cup \{s\}$.

Let $K \in \Sub(G/H)$ and let $\mc{T}_K = \bigcup_{k \in K}k\mc{L}$.  Since $\xi$ is fixed, all the lines $k\mc{L}$ have a common end $\xi$, ensuring that $\mc{T}_K$ is connected and hence a subtree of $\mc{T}$.  Since $\mc{L}$ is the axis of translation for $s \in K$, we see that for all $k \in K$, then $k\mc{L}$ is the axis of the translation $ksk\inv \in K$.  Thus $\mc{T}_K$ is a union of axes of translation of $K$; since $\mc{T}_K$ is $K$-invariant, it is the smallest $K$-invariant subtree of $\mc{T}$.  This proves (i).

Given that $X = \partial \mc{T} \setminus \{\xi\}$ is $G$-equivariantly homeomorphic to $G/H$, we see that the $K$-orbit $K\xi'$ is closed in $X$, so $K\xi' \cup \{\xi\}$ is closed in $\partial T$.  Note that any closed subset of $\partial T$ that is not a singleton is the set of ends of some subtree of $T$; the minimality of $\mc{T}_K$ ensures that $\{\xi\}$ is maximal among proper closed $K$-invariant subsets of $\partial \mc{T}_K$, so in fact $K\xi' \cup \{\xi\} = \partial T_K$.  In particular, every end of $\mc{T}_K$ is an end of translation of $K$, which means that the set $\{k\mc{L} \mid k \in K\}$ is the set of all bi-infinite lines through $\mc{T}_K$ ending at $\xi$.  In particular, given $K,K' \in \Sub(G/H)$, we have
\[
\mc{T}_K \subseteq \mc{T}_{K'} \Leftrightarrow \{k\mc{L} \mid k \in K\} \subseteq \{k\mc{L} \mid k \in K'\} \Leftrightarrow K \le K',
\]
where the second equivalence is due to the fact that $K$ and $K'$ both contain the setwise stabilizer $H$ of $\mc{L}$ in $G$.  This proves (ii).  If $K \in \Sub(G/H) \setminus \{H\}$, then $\mc{T}_K$ properly contains $\mc{L}$.  In addition, $K$ acts with minimum translation length $l$ since this is the translation length of $s$ and also the minimum translation length of $G$.  As we have noted, $K$ acts transitively on $\partial \mc{T}_K \setminus \{\xi\}$; hence Lemma~\ref{lem:tree_end_types} applies to the action $(\mc{T}_K,K)$.  We then see that all types are ruled out except focal type, proving (iii).

Our next aim is to show that no element of $\Sub(G/H) \setminus \{H\}$ is $U$-thin over $H$, where $U$ is a compact open subgroup of $G$ containing $U_0$.  By Proposition~\ref{prop:thick} we have
\[
T_U(H) = \bigcap_{h \in H}hUH = \bigcap_{n \in \Zb}s^nU\grp{s};
\]
we claim that in fact $H = T_U(H)$, for which it is enough to show $T_H(U) \cap U = H \cap U$. Consider $T_n = s^nU\grp{s} \cap U$ for $n \in \Zb$.  Then $T_n$ does not contain any translations and hence stabilizes each horosphere around $\xi$; thus $T_n = s^nUs^{-n} \cap U$.  Now letting $n$ vary over $\Zb$, we see that
\[
T_U(H) \cap U = \bigcap_{n \in \Zb}(s^nUs^{-n} \cap U) \le \bigcap_{n \in \Zb}s^nUs^{-n} = U_0 \le H,
\]
as claimed.  Note that the fact that $U$ is open and contains $U_0$ means that $U$ contains $G_w$ for some $w \in V\mc{L}$.

Now consider $K \in \Sub(G/H)$ such that $K/H$ is $U$-thin; by Lemma~\ref{lem:thin_criterion}, we have $K \cap HUH = H$, so in particular $K_w \le H$.  Consider now a $K$-translate $k\mc{L}$ of $\mc{L}$ containing the vertex $w$.  After replacing $k$ with $ks^n$ for some $n \in \Zb$ we may assume that $k$ is an elliptic element of $K$; the fact that $k\mc{L}$ and $\mc{L}$ both contain $w$ then ensures $kw = w$, so $k \in H$ and hence $k\mc{L} = \mc{L}$.  Thus $\xi'$ is an isolated point in $K\xi'$; by part (iii) we deduce that $\partial \mc{T}_K = \{\xi,\xi'\}$ and hence $K = H$.  So no element of $\Sub(G/H) \setminus \{H\}$ is $U$-thin over $H$.

Given $K \in \Sub(G/H) \setminus \{H\}$ and $v \in V\mc{L}$, we have seen that $K/H$ is not $G_v$-thin.  Applying Lemma~\ref{lem:thin_criterion} again we see that
\[
K \cap (sG_v \cup s\inv G_v) \not\subseteq G_vH;
\]
since $sG_vs\inv$ contains $G_v$ and $s \in H$, we can rearrange to get
\[
K \cap G_v \not\subseteq s\inv G_vH.
\]
In other words, $K_v$ does not fix $s\inv v$.  This completes the proof of (iv).

Part (v) now follows from (iv) using Lemma~\ref{lem:descending_chain}.
\end{proof}

Proposition~\ref{prop:focal_almost_transitive}(v) shows that $(\mc{T},G)$ involves another almost boundary-transitive focal action $(\mc{T}_K,K)$ that is ``irreducible'' in the sense that $H = K_{\xi'}$ is a maximal closed subgroup of $K$ for every $\xi' \in \partial \mc{T}_K \setminus \{\xi\}$.  In general this does not produce a complete $\tCDI$ factorization of $G/H$, since one still has to factorize $G/K$.  Nevertheless, the general arguments for coset spaces still apply, and one can continue to analyse the groups $K' \in \Sub(G/K)$ by looking at their minimal trees $\mc{T}_{K'}$.  Proposition~\ref{prop:focal_almost_transitive}(iv) suggests an approach based on local actions, namely one use the action of $K'_v$ (for $v \in V\mc{L}$) on a ball of radius $l$ to distinguish the possibilities for $K'$ in an ascending chain.

\section{Irreducible coset spaces}\label{sec:irreducible}

\subsection{Introduction}

In the context of locally compact groups, it is natural to consider irreducible coset spaces $G/H$, where $G$ is a locally compact group and $H$ is a subgroup that is maximal among proper closed subgroups of $G$.  Theorem~\ref{thm:cdi}, Proposition~\ref{prop:cdi_maximal_in_G} and Corollary~\ref{cor:cdi_maximal} show that such coset spaces are important for understanding how compactly generated closed subgroups are embedded in a general locally compact group, and we can expect them to appear in many contexts.  In the classification of such coset spaces, it is convenient to regard $G/H$ and $(G/K)/(H/K)$ as equivalent, where $K = \bigcap_{g \in G}gHg\inv$, since the closed normal subgroup $K$ effectively makes no contribution to the properties of the action of $G$ on $G/H$; thus when discussing irreducible coset spaces, we will often assume that $G/H$ is \defbold{faithful}, that is, $\bigcap_{g \in G}gHg\inv = \triv$.  The result is that we are considering something similar to a primitive permutation group, except that the topology of $G$ is not necessarily the permutation topology (in particular, $H$ need not be open), and $H$ is not necessarily maximal among all proper subgroups of $G$ (in other words, $H$ could be contained in a proper dense subgroup).

A typology was recently established by Smith \cite{SmithPrimitive} of the closed primitive subdegree-finite permutation groups, generalizing the O'Nan--Scott theorem to this setting; in our terms, these correspond to the faithful irreducible coset spaces $G/H$ where $H$ is a compact open subgroup.  Smith's proof is by induction on the minimum nontrivial subdegree, in other words the minimum of $|H:H \cap gHg\inv|$ for $g \in G \setminus H$.  By contrast, in the context of the present article, we are rather more interested in the case that $H$ is neither compact nor open.  This presents the difficulties that there is no obvious basis of induction, and there are additional topological complications: for example, if $H$ is neither compact nor open, then the product of $H$ with a closed normal subgroup could be a proper dense subgroup of $G$.

In this section we begin the analysis of faithful irreducible coset spaces without the assumption of primitivity, with a focus on \lcsc groups that are compactly generated.  By analogy with the analysis of primitive permutation groups, we will mostly focus on the role of nontrivial closed normal subgroups, although compared to prior work on primitive permutation groups, we will be focusing on some rather basic questions, such as the number of minimal nontrivial closed normal subgroups.

\subsection{General observations}\label{sec:irreducible:general}

For this section we consider the situation where $G/H$ is a faithful irreducible coset space, but $G$ and $G/H$ are not necessarily locally compact.  Let us introduce some terminology.

\begin{defn}
Given a coset space $G/H$, say a closed normal subgroup $N$ of $G$ is \defbold{$G/H$-free} if $N \cap H = \triv$, a \defbold{$G/H$-supplement} if $NH = G$, and a \defbold{$G/H$-complement} if $N \cap H =\triv$ and $G = NH$.

Given a Hausdorff topological group $G$, write $\mc{M}_G$ for the set of minimal nontrivial closed normal subgroups of $G$.  We say $G$ is \defbold{monolithic} if the intersection $\Mon(G)$ of all nontrivial closed normal subgroups of $G$ is nontrivial.
\end{defn}

If $G$ is Polish and there is a closed normal $G/H$-complement $M$, we observe that $G = M \rtimes H$ as a topological group.

Let $G/H$ be a faithful irreducible coset space and let $N$ be a nontrivial closed normal subgroup.  We see that $G = \ol{NH}$, so $N$ is a $G/H$-supplement if and only if $NH$ is closed.  It follows that under certain hypotheses, $N$ is automatically a $G/H$-supplement: for example, this is the case if $N$ or $H$ is compact, or if $N$ or $H$ is open.

We can say the following about closed $H$-invariant subgroups of $G$.

\begin{lem}\label{lem:irreducible_invariant}
Let $G/H$ be a faithful irreducible coset space and let $K$ be a nontrivial closed $H$-invariant subgroup of $G$.  Then exactly one of the following holds:
\begin{enumerate}[(i)]
\item We have $K \le H = \N_G(K)$;
\item $K$ is a normal subgroup of $G$ that is not contained in $H$.
\end{enumerate}
\end{lem}

\begin{proof}
Clearly (i) and (ii) are mutually exclusive, and if $\N_G(K) = H$ then $K$ is a normal subgroup of $H$, so we may assume $\N_G(K) \nleq H$.  In that case, since $K$ is $H$-invariant, we then have $H < \N_G(K)$; since $K$ is closed, $\N_G(K)$ is a closed subgroup, so $\N_G(K)=G$; and then since $G/H$ is faithful, $K \nleq H$.
\end{proof}

In particular, given $M \in \mc{M}_G$, we can apply the dichotomy of Lemma~\ref{lem:irreducible_invariant} to $H$-invariant subgroups of $M$.

\begin{cor}\label{cor:irreducible_invariant}
Let $G/H$ be a faithful irreducible coset space and let $K$ be a closed normal subgroup of $G$.  Then $\bigcap_{k \in K}k(K \cap H)k\inv = \triv$.  If $K \in \mc{M}_G$, then $K \cap H$ is the unique largest proper closed $H$-invariant subgroup of $K$.
\end{cor}

\begin{proof}
We see that the intersection $L = \bigcap_{k \in K}k(K \cap H)k\inv$ is a closed normal subgroup of $H$ such that $\N_G(L) \nleq H$, so by Lemma~\ref{lem:irreducible_invariant}, $L = \triv$.

Now suppose $K \in \mc{M}_G$.  Given a proper nontrivial $H$-invariant subgroup $R$ of $K$, then $R$ cannot be normal in $G$ by hypothesis, so by Lemma~\ref{lem:irreducible_invariant}, $R \le H$ and hence $R \le K \cap H$.  On the other hand, $K \cap H$ is itself clearly a closed $H$-invariant subgroup of $K$, and we have $K \cap H < K$ since $H$ does not contain any nontrivial normal subgroup of $G$.
\end{proof}

We can say a bit more about which subgroups are $G/H$-free.  In particular, we see that every closed normal subgroup that is not $G/H$-free must have trivial centralizer.

\begin{lem}\label{lem:irreducible_free}
Let $G/H$ be a faithful irreducible coset space.
\begin{enumerate}[(i)]
\item Let $N$ be a nontrivial closed normal subgroup of $G$.  Then $\CC_H(N) = \triv$.
\item Let $K$ be a closed subgroup of $G$ such that $\bigcap_{h \in H}\CC_G(hKh\inv) \nleq H$.  Then $K \cap H = \triv$.
\item Let $N$ be a closed normal subgroup of $G$ such that $\CC_G(N) \neq \triv$.  Then $N \cap H = \triv$.
\item Every normal subgroup of $G$ with nontrivial centre is $G/H$-free.
\item For every proper subset $\mc{M}'$ of $\mc{M}_G$, we have $\cgrp{N \in \mc{M}'} \cap H = \triv$.
\end{enumerate}
\end{lem}

\begin{proof}
(i)
We see that $\CC_H(N)$ is normalized by both $N$ and $H$, so it is normal in $\ol{NH} = G$.  Since $G/H$ is faithful for $G$, it follows that $\CC_H(N)=\triv$.

(i)
Let $C = \bigcap_{h \in H}\CC_G(hKh\inv)$.  Then $C$ is a closed $H$-invariant subgroup of $G$, hence a normal subgroup of $G$ by Lemma~\ref{lem:irreducible_invariant}, so $\CC_H(C) = \triv$ by (i).  Clearly $K$ commutes with $C$, so we have $K \cap H = \triv$.

(iii) follows immediately from (i) and (ii).

(iv)
Let $N$ be a normal subgroup of $G$ with nontrivial centre $Z$; note that $Z$ is central in $\ol{N}$ and normal in $G$.  Then $\ol{N} \cap H$ is a closed normal subgroup of $H$ that is also normalized by $Z$; however, $G = \ol{ZH}$, so in fact $\ol{N} \cap H$ is normalized by $G$, and hence $\ol{N} \cap H = \triv$.  In particular, $N \cap H = \triv$, so $N$ is $G/H$-free.

(v)
We see that distinct elements of $\mc{M}_G$ centralize each other, so if $\mc{M}'$ is a proper subset of $\mc{M}_G$, say $M \in \mc{M}_G \setminus \mc{M}'$, then $M$ commutes with $K = \cgrp{N \in \mc{M}'}$.  It then follows by (ii) that $K$ has trivial intersection with $H$.
\end{proof}

If there is a minimal closed normal subgroup that is $G/H$-free, we can regard $G/H$ as a ``compression'' of a faithful irreducible coset space coming from a semidirect product.

\begin{prop}\label{prop:irreducible_free:semidirect}
Let $G$ be a Hausdorff topological group, let $N$ be a closed normal subgroup of $G$ and let $H$ be a closed subgroup of $G$.
\begin{enumerate}[(i)]
\item Suppose that $G = N \rtimes H$.  Then $G/H$ is a faithful irreducible coset space if and only if $\CC_H(N) = \triv$ and there is no proper nontrivial closed $H$-invariant subgroup of $N$.
\item Suppose that $G = N \rtimes H$, that $G/H$ is a faithful irreducible coset space and $\CC_G(N) > \triv$.  Then there is a continuous injective homomorphism from $\CC_G(N)$ to $N$ with dense image.
\item Suppose that $M \in \mc{M}_G$, that $G/H$ is a faithful irreducible coset space and that $M \cap H = \triv$; let $\alpha$ be the conjugation action of $H$ on $M$.  Then $(M \rtimes_{\alpha} H)/H$ is a faithful irreducible coset space, and the map $(m,h) \mapsto mh$ is a continuous injective homomorphism from $M \rtimes H$ to $G$ with dense image.
\end{enumerate}
\end{prop}

\begin{proof}
Consider first the case $G = N \rtimes H$.  We see that $\bigcap_{g \in G}gHg\inv = \CC_H(N)$, so $G$ acts faithfully on $G/H$ if and only if $\CC_H(N)  = \triv$.  We then see that every closed overgroup $L$ of $H$ is of the form $(L \cap N)H$, where $L \cap N$ is closed and $H$-invariant; conversely, given a closed $H$-invariant subgroup $R$ of $N$, then $L = R \rtimes H$ is a closed overgroup of $H$ and $R = L \cap N$.  Thus $H$ is maximal among proper closed subgroups of $G$ if and only if there is no proper nontrivial closed $H$-invariant subgroup of $N$.  This proves (i).

For (ii), we assume in addition that $G/H$ is a faithful irreducible coset space and $\CC_G(N) > \triv$.  Let $\alpha: G \rightarrow \Aut(N)$ be the conjugation action.  Then $\CC_G(N)$ consists of pairs $(n\inv,h)$ where $\alpha(h) = \alpha(n)$.  In particular, since $H$ acts faithfully on $N$ by conjugation, we see that the map
\[
\theta: \CC_G(N) \rightarrow N; \; (n\inv,h) \mapsto n
\]
is injective; clearly $\theta$ is continuous, and an easy calculation shows that $\theta$ is a homomorphism.  The image $\theta(\CC_G(N))$ is $H$-invariant; thus it is dense by part (i).  This completes the proof of (ii).

For (iii), we assume instead that $G$ is a Hausdorff topological group with $M \in \mc{M}_G$ and a closed subgroup $H$ such that $G/H$ is a faithful irreducible coset space.  We have $\CC_H(M) = \triv$ by Lemma~\ref{lem:irreducible_free}(i), and by Corollary~\ref{cor:irreducible_invariant} there is no proper nontrivial closed $H$-invariant subgroup of $M$.  Thus $(M \rtimes_{\alpha} H)/H$ is a faithful irreducible coset space by (i).  The map $(m,h) \mapsto mh$ is a continuous injective homomorphism from $M \rtimes H$ to $G$ with image $MH$; since $MH$ is a subgroup of $G$ properly containing $H$, it is dense.  This completes the proof of (iii).
\end{proof}

In the situation that Proposition~\ref{prop:irreducible_free:semidirect}(iii) applies, we will say that $G/H$ is of \defbold{compressed semidirect type} and $(M \rtimes_{\alpha} H)/H$ is a \defbold{irreducible semidirect form} of $G/H$.  Note that irreducible semidirect forms are not necessarily unique, since we can have $|\mc{M}_G|>1$.

Given a topologically characteristically simple group $M$, we can certainly form the faithful irreducible coset space $(M \rtimes H)/H$, where $H = \Aut(M)$ with the discrete topology.  So every topologically characteristically simple group can appear as the group $M$ in Proposition~\ref{prop:irreducible_free:semidirect}(iii).  Moreover, if $M = A/B$ occurs as a chief factor of some topological group $K$, then $K$ acts on $A/B$ by conjugation, with kernel $\CC_K(M)$, and we again have a faithful irreducible coset space $G/H$ where $H = K/\CC_K(M)$ and $G = M \rtimes H$.  In the latter case, if $K$ is Polish, locally compact, or compactly generated locally compact, then so is $G$.  There are then other possibilities for faithful coset spaces $(M \rtimes H)/H$, where $H$ does not induce any inner automorphisms of $M$, but nevertheless acts in such a way that there is no proper nontrivial closed $H$-invariant subgroup of $M$.  Given the apparent diversity of chief factors of \lcsc groups, even in the compactly generated case, and the limited knowledge of the automorphism groups of the topologically characteristically simple groups occurring in this context, there is not much more we can say at this level of generality.

Note that (compressed) semidirect type account for all faithful irreducible coset spaces $G/H$ where $G$ has an abelian minimal closed normal subgroup.

\begin{cor}\label{cor:irreducible_abelian}
Let $G/H$ be a faithful irreducible coset space and suppose that $M \in \mc{M}_G$ is abelian.  Then $G/H$ is of compressed semidirect type with an irreducible semidirect form $(M \rtimes H)/H$.
\end{cor}

\begin{proof}
By Lemma~\ref{lem:irreducible_free}(iv) we have $M \cap H = \triv$; hence by Proposition~\ref{prop:irreducible_free:semidirect}, the semidirect product $M \rtimes H$ yields a faithful irreducible coset space $(M \rtimes H)/H$, and we have a continuous injective homomorphism from $M \rtimes H$ to $G$ with dense image.
\end{proof}

\begin{rmk}
In the case where $M \in \mc{M}_G$ is an abelian \lcsc group, the possibilities for $M$ as a topological group were classified in \cite{Reid_abelian}; they can be listed as
\[
C^{(\kappa)}_p; \; C^{\aleph_0}_p; \Fb_p((t));  \; \Rb^n; \; \Qb^{(\kappa)}; \; \widehat{\Qb}^{\kappa}; \; \Qb^n_p; \; \Qb_p(\aleph_0),
\]
where $p$ is a prime number and $1 \le \kappa \le \aleph_0$.  Here $A^{(\kappa)}$ is a direct sum and $A^{\kappa}$ is a direct product of $\kappa$ copies of $A$; $\widehat{\Qb}$ is the Pontryagin dual of $\Qb$; and $\Qb_p(\aleph_0)$ is the set of functions from $\Nb$ to $\Qb_p$ with all but finitely many values in $\Zb_p$, equipped with pointwise multiplication and with $\Zb^{\aleph_0}_p$ embedded as a compact open subgroup.  The structure of $C^{(\kappa)}_p$, $\Rb^n$, $\Qb^{(\kappa)}$ and $\Qb^n_p$ as abelian topological groups is described by regarding them as vector spaces over the appropriate locally compact field. The groups $\widehat{\Qb}^{\kappa}$ and $C^{\aleph_0}_p$ have the same automorphism groups as $\Qb^{(\kappa)}$ and $C^{(\aleph_0)}_p$ respectively, via Pontryagin duality. The group $\Qb_p(\aleph_0)$, is a $\Zb_p$-module rather than a $\Qb_p$-vector space, but as a locally compact abelian group, it is self-dual and can even be thought of as a kind of ``separable $p$-adic Hilbert space'': see \cite{CT}; one can perhaps regard $\Fb_p((t))$ as something analogous in positive characteristic.  So the classification of faithful irreducible coset spaces of the form $(M \rtimes H)/H$, where $M$ is an abelian \lcsc group, essentially becomes a problem of representation theory.  It is less clear how to classify the compressed forms of $(M \rtimes H)/H$, however.
\end{rmk}

\subsection{Compact normal subgroups}

We can put some restrictions on the compact normal subgroups of a faithful irreducible coset space.

\begin{lem}\label{lem:irreducible:minimal_compact}
Let $G/H$ be a faithful irreducible coset space.
\begin{enumerate}[(i)]
\item There is a neighbourhood $O$ of the identity that does not contain any nontrivial normal subgroup of $G$.
\item Every nontrivial compact normal subgroup of $G$ contains some $M \in \mc{M}_G$.
\item Every nontrivial compact normal subgroup of $G$ is a $G/H$-supplement.
\end{enumerate}
\end{lem}

\begin{proof}
(i)
The coset space $G/H$ is a Hausdorff space with more than one point, so it has an open neighbourhood of the trivial coset whose closure is not all of $G/H$.  We can write this open neighbourhood as $O/H$, where $O$ is an open subset of $G$ formed as a union of left $H$-cosets; in particular $O$ is an identity neighbourhood, whose closure is not all of $G$.  Since $H$ is a maximal closed subgroup, it follows that if $K$ is any (not necessarily closed) subgroup of $G$ such that $H \le K \subseteq O$, then $K = H$.  In particular, if $N$ is a normal subgroup of $G$ such that $N \subseteq O$, then $NH \subseteq O$, from which it follows that $NH = H$; since $\bigcap_{g \in G}gHg\inv = \triv$, it follows that $N = \triv$.

(ii)
By Zorn's lemma, it suffices to show that the intersection $K$ of any descending chain $(K_i)_{i \in I}$ of nontrivial compact normal subgroups is nontrivial.  Indeed, by part (i), we see that given such a chain, for each $i \in I$ the set $K_i \setminus O$ is nonempty.  Moreover, $K_i \setminus O$ is compact, since it is an intersection of a compact set with a closed set; thus $K \setminus O$ is nonempty and hence $K$ is nontrivial.

(iii)
Given a nontrivial compact normal subgroup $K$ of $G$, we see that $K \nleq H$ and $KH$ is closed; thus $KH = G$.
\end{proof}

\begin{cor}\label{cor:irreducible:minimal_compact}
Let $G/H$ be a faithful irreducible coset space and suppose $G$ has a nontrivial compact normal subgroup $N$.  Then $N$ contains some $M \in \mc{M}_G$ and $G = MH$.  If $M \cap H = \triv$, then $G \cong M \rtimes H$ as topological groups.
\end{cor}

\begin{proof}
By Lemma~\ref{lem:irreducible:minimal_compact} we see that $N$ contains some $M \in \mc{M}_G$ and we have $G = MH$.

Suppose now that $M \cap H = \triv$.  Then clearly we have a continuous bijective homomorphism $\theta: M \rtimes H \rightarrow G$ given by $\theta(m,h) = mh$.  Consider now a neighbourhood $O$ of the identity in $M \rtimes H$: to show $\theta$ is a homeomorphism, it is enough to show $\theta(O)$ is an identity neighbourhood in $G$.  So suppose for a contradiction there is a net $(g_i)$ in $G$ converging to the identity, such that $g_i \not\in \theta(O)$ for all $i$.  Then we can write $g_i = m_ih_i$ for $m_i \in M$ and $h_i \in H$, and after passing to a subnet, $(m_i)$ converges to some $m \in M$.  Then $(h_i)$ converges to $m\inv$, and since $M \cap H = \triv$ we have $m = 1$.  So the sequences $(m_i)$ and $(h_i)$ both converge to the identity, which means $(m_i,h_i) \in O$ for some $i$, a contradiction.  Thus $\theta$ is an isomorphism of topological groups.
\end{proof}

The following is immediate from Lemma~\ref{lem:irreducible:minimal_compact}(i) and Corollary~\ref{cor:second_countable}.

\begin{cor}\label{cor:irreducible:second_countable}
Let $G/H$ be a faithful irreducible coset space, such that $G$ is a $\sigma$-compact locally compact group.  Then $G$ is second-countable.
\end{cor}

The compact second-countable topologically characteristically simple groups have been classified: they are all either abelian or of semisimple type, and those of semisimple type can be given as a direct product $S^{\kappa}$, where $1 \le \kappa \le \aleph_0$ and $S$ belongs either to the class of nonabelian finite simple groups or the class of abstractly simple connected compact Lie groups, both classes that have been completely described.

\subsection{Compactly generated locally compact groups}\label{sec:cg_irreducible}

In the setting of compactly generated locally compact groups, we have strong restrictions on descending chains of closed normal subgroups of $G$; see \cite{RW-EC}.  In the case of a faithful irreducible coset space, we deduce the following special case.

\begin{lem}\label{lem:minimal_normal}
Let $G$ be a compactly generated locally compact group and let $H \le G$ be such that $G/H$ is a faithful irreducible coset space.  Let $N$ be a nontrivial closed normal subgroup of $G$.  Then at least one of the following holds:
\begin{enumerate}[(i)]
\item There is $M \le N$ such that $M \in \mc{M}_G$;
\item There is a descending chain $(D_i)_{i \in I}$ of nontrivial discrete normal subgroups of $G$ contained in $N$, such that $\bigcap_{i \in I}D_i = \triv$.
\end{enumerate}
\end{lem}

\begin{proof}
We suppose that (i) is not the case and let $G$ act on itself by conjugation.  By Zorn's lemma, there is a descending chain $(D_i)_{i \in I}$ of nontrivial closed $G$-invariant subgroups of $N$ with trivial intersection.  By \cite[Theorem~3.3]{RW-EC}, there is some $i \in I$ for which there exists a compact open $G$-invariant subgroup $K_i$ of $D_i$.  However, by Lemma~\ref{lem:irreducible:minimal_compact}, $N$ does not contain any nontrivial compact normal subgroup of $G$, so $K_i = \triv$, and hence $D_i$ is discrete.  Any subgroup of $D_i$ is discrete, so by replacing $(D_i)_{i \in I}$ with a subsequence, we obtain a descending chain as in (ii).
\end{proof}

We now prove the theorem from the introduction.

\begin{proof}[Proof of Theorem~\ref{thm:cg_irreducible}]
We first note that $G$ is $\sigma$-compact since it is compactly generated; it is therefore second-countable by Corollary~\ref{cor:irreducible:second_countable}.  For every nontrivial normal subgroup $N$ of $G$, we have $\CC_H(N) =\triv$ by Lemma~\ref{lem:irreducible_free}(i).

Suppose $\mc{M}_G = \emptyset$.  Then (iv) follows by Lemma~\ref{lem:minimal_normal}, while (i), (ii) and (iii) are clearly false.  From now on we may assume $\mc{M}_G \neq \emptyset$; this rules out (iv), and the remaining cases are clearly mutually exclusive.

Suppose there is $M \in \mc{M}_G$ with $\CC_G(M) = \triv$.  Then every nontrivial closed normal subgroup intersects $M$ nontrivially, and hence contains $M$.  Thus in fact $M = \Mon(G)$ and we are in case (iii).  Hence we may suppose for all $M \in \mc{M}_G$ that $\CC_G(M) \neq \triv$.  We then deduce from Lemma~\ref{lem:irreducible_free}(iii) that $M \cap H = \triv$ for every $M \in \mc{M}_G$.  Applying Proposition~\ref{prop:irreducible_free:semidirect}, we deduce that $G/H$ is of compressed semidirect type.  If $G$ has no nontrivial compact normal subgroup, then (ii) holds.  Otherwise, by Lemma~\ref{lem:irreducible:minimal_compact}(ii), some $M \in \mc{M}_G$ is compact, and applying Corollary~\ref{cor:irreducible:minimal_compact}, we have $G = M \rtimes H$; we then deduce case (i) via Proposition~\ref{prop:irreducible_free:semidirect}(ii).
\end{proof}

Supposing we have a faithful irreducible coset space $(M \rtimes H)/H$, then if $M$ is not compact, there could be various topological groups $G$ and continuous injective homomorphisms $\theta: M \rtimes H \rightarrow G$, restricting to closed maps on $M$ and on $H$, such that the image is dense and $G/\theta(H)$ is a faithful irreducible coset space.  Here are some examples in the setting of compactly generated \tdlc groups.

\begin{example}\label{ex:no_complement}
\begin{enumerate}[(i)]
\item
Let $S$ be a nonabelian finite simple group.  Let $M_1$ be the direct product $S^{\Zb}$ and let $M_2$ be the direct sum $\bigoplus_{\Zb}S$ with the discrete topology; note that there is an obvious continuous dense injective homomorphism $\theta$ from $M_2$ to $M_1$.  Via $\theta$, we have a faithful action of $M_2$ on $M_1$ by conjugation; this action naturally extends to an action $\alpha$ of $H = M_2 \rtimes_{\sigma} \Zb$, where $\sigma$ is the shift action of $\Zb$.  Form the semidirect product $G = M_1 \rtimes_{\alpha} H$; observe that $H$ leaves no proper nontrivial closed subgroup of $M_1$ invariant and $\CC_H(M_1) = \triv$, so $G/H$ is a faithful irreducible coset space.  We see that $M_3 := \CC_G(M_1)$ is a closed normal subgroup, given by
\[
M_3 = \{(m\inv_1,m_2) \in M_1 \rtimes_{\alpha} M_2 \mid \theta(m_2) = m_1\}.
\]
It is then easy to verify that $M_3 \cong M_2$ as topological groups, so $M_3 \not\cong M_1$, and moreover $M_3 \in \mc{M}_G$ and $M_3 \cap H = \triv$.  We now observe that the subgroup $M_3H$ of $G$ is a proper dense subgroup of $G$, abstractly isomorphic to a semidirect product $M_3 \rtimes H$, but the topology of $M_3H$ is coarser than that of $M_3 \rtimes H$; thus there is an injective continuous homomorphism from $M_3 \rtimes H$ to $G$ with dense image.  The coset space $(M_3 \rtimes H)/H$ is itself irreducible.  Finally, observe that $M_3 \rtimes H$ is a finitely generated discrete group, while $G$ is a compactly generated nondiscrete \tdlc group.

\item
With a little more work, we can obtain a faithful irreducible coset space with no closed normal complement.  Let $p$ and $q$ be distinct prime divisors of $|S|$ and for $r \in \{p,q\}$, let $T_r$ be an $r$-Sylow subgroup of $S$ and let $L_r$ be the local direct product $\bigoplus_{\Zb}(S,T_r)$.  As before we take $H = M_2 \rtimes_{\sigma} \Zb$.  Let $\theta_r$ be the obvious continuous dense injective homomorphism from $M_2$ to $L_r$; let $M_2$ act on $L_r$ via the conjugation action of $\theta_r(M_2)$; combine the actions to give a coordinatewise action of $M_2$ on $L_p \times L_q$, and then extend in the obvious way to an action $\beta$ of $H$ on $L_p \times L_q$.  Note that there is a closed subgroup $D$ of $L_p \times L_q$ that is a diagonal copy of $M_2$: given a net $(m_i)$ in $M_2$, then since $T_p \cap T_q = \triv$, the net $(\theta_p(m_i),\theta_q(m_i))$ converges in $L_p \times L_q$ if and only if it is eventually constant in each coordinate, that is, $(m_i)$ converges in $M_2$.  Form the semidirect product $G_1 = (L_p \times L_q) \rtimes_{\beta} H$.  We then form a Hausdorff group quotient $\pi: G_1 \rightarrow G_2$ of $G_1$, where $\pi$ is chosen so that $\ker\pi$ consists of all triples $(l\inv_p,l\inv_q,h)$ such that $\beta(h)$ acts as conjugation by $(l_p,l_q)$ on $L_p \times L_q$.  (One can check that the group just described as $\ker\pi$ is indeed a closed normal subgroup of $G_1$.)  We see that $\ker\pi \cap H = \triv$, and the product $\ker\pi H$ is closed: specifically, $\ker\pi H = D \rtimes H$.  Thus $\pi(H)$ is closed, and in fact, since we are working with Polish groups, it follows that $\pi$ restricts to a closed embedding of $H$ into $G_2$.  Similarly, $\pi$ restricts to closed embeddings of $L_p$ and $L_q$ into $G_2$, and we see that $\pi(H)$ acts faithfully on $\pi(L_p)$ and on $\pi(L_q)$ by conjugation and has trivial intersection with each of them.  Any proper overgroup $K$ of $\pi(H)$ in $G_2$ will contain the image of some element $(l\inv_p,l\inv_q,h)$ such that for at least one $r \in \{p,q\}$, $\theta_{r}(h)$ does not act as $l_r$; from there, one can deduce that $\pi(L_{r}) \le K$ and hence that $K$ is dense.  We now see that $G_2/\pi(H)$ is a faithful irreducible coset space.  We also see that $\mc{M}_G = \{\pi(L_p),\pi(L_q)\}$, and both $\pi(L_p)$ and $\pi(L_q)$ are $G_2/\pi(H)$-free; however, the product $\pi(L_p)\pi(H)$ does not contain $\pi(L_q)$ and \textit{vice versa}.  By considering the closed normal subgroups of $G_1$, we see that any nontrivial closed normal subgroup of $G_2$ other than $\pi(L_p)$ and $\pi(L_q)$ contains the subgroup $\pi(M_2)$ of $\pi(H)$.  Thus there are two minimal closed normal $G_2/\pi(H)$-free subgroups, but no closed normal $G_2/\pi(H)$-complement.
\end{enumerate}
\end{example}

\subsection{Normal subgroups of semisimple type}\label{sec:semisimple}

Let $N$ be a topologically characteristically simple Polish group.  We say $N$ is of \defbold{(strict) semisimple type} if $N$ has a closed normal subgroup $S$ that is nonabelian and topologically simple.  It follows that $N$ is topologically generated by the set $\mc{S} := \{\alpha(S) \mid \alpha \in \Aut(N)\}$ of $\Aut(N)$-conjugates of $S$, which are the \defbold{components} of $N$.  By \cite[Lemma~5.14]{RW-Polish} we have $|\{\alpha(S) \mid \alpha \in \Aut(N)\}| \le \aleph_0$.  On the other hand, any closed normal subgroup of $N$ that centralizes $\alpha(S)$ for all $\alpha \in \Aut(N)$ is central and hence trivial.  By \cite[Theorem~1.7]{RW-Polish}, the diagonal map $d: N \rightarrow \prod_{Q \in \mc{S}}N/\CC_N(Q)$ is a continuous injective homomorphism; moreover, by \cite[Theorem~1.4]{RW-Polish}, each of the groups $N/\CC_N(Q)$ is topologically simple, and we see that $N/\CC_N(\alpha(S))$ is isomorphic to $N/\CC_N(S)$ for all $\alpha \in \Aut(N)$.  See \cite{RW-Polish} for more information.

Suppose now that $G$ is a Polish group, with $H$ a closed subgroup, such that there exists $M \in \mc{M}_G$ of semisimple type, with topologically simple closed normal subgroup $S$; write $K = M \cap H$.  Let us consider some necessary conditions for $G/H$ to be faithful irreducible.  We need the product $MH$ to be dense in $G$.  (Actually, we need the condition that for every proper closed overgroup $H'$ of $H$, then $M \cap H' > K$; but this is difficult to characterize in full generality: it is stronger than the condition that $MH$ is dense, but weaker than requiring $MH = G$.)  The fact that $M$ is a minimal closed normal subgroup of $G$ is then equivalent to the condition 
\[
\{\alpha(S) \mid \alpha \in \Aut(M)\} = \{hSh\inv \mid h \in H\}.
\]
We need to have $\bigcap_{m \in M}mKm\inv = \triv$, which now simply means that $S \nleq K$.  The final condition is that $K$ should be the unique largest proper closed $H$-invariant subgroup of $M$; this last condition is much more difficult to describe.

Already for finite groups, in the notation introduced by C. Praeger in \cite{Praeger}, there are seven types of primitive action that can occur (all except type HA).  We will not attempt to classify all the types of (locally compact) Polish groups that can occur in this context, but the following discussion should at least illustrate how analogues of all the O'Nan--Scott types also occur for infinite locally compact groups and Polish groups, and also some topological obstacles to their construction.  Let $\alpha: H \rightarrow \Aut(M)$ be the conjugation action of $H$ on $M$ and let $\beta: M \rightarrow \Inn(M)$ be the conjugation action of $M$ on itself.  When we refer to the type of a (not necessarily faithful) irreducible coset space, we mean the type of its faithful form.

\paragraph{\bf Type HS}
In this case, $M=S$ and $K = \triv$.  The group $H$ is such that $\beta\inv(\alpha(M))$ is a dense subgroup of $M$: this means that every closed $H$-invariant subgroup of $M$ is $M$-invariant, and hence either $\triv$ or $M$.  Note that $G$ is not monolithic in this case: the centralizer $\CC_G(M)$ contains the nontrivial group $R$ of all elements $m\inv h$ such that $m \in M$, $h \in H$ and $\beta(m) = \alpha(h)$.  In the case that $G = M \rtimes H$, the centralizer of $M$ is exactly $R$.

\paragraph{\bf Type HC}  
Again $K = \triv$, but now $S < M$, so $|\{\alpha(S) \mid \alpha \in \Aut(M)\}| \ge 2$.  This time $H$ is such that $(S \rtimes \N_S(H))/\N_S(H)$ is of type HS and $H$ acts transitively on $\{\alpha(S) \mid \alpha \in \Aut(M)\}$.  As for type HS, the centralizer of $M$ is nontrivial, but not necessarily isomorphic to $M$ as a topological group.  Example~\ref{ex:no_complement}(i) is of type HC.

\

Clearly, for every topologically characteristically simple Polish group $M$ of semisimple type, there are examples of Polish faithful irreducible coset spaces $G/H$ of type HS or HC that have $M \in \mc{M}_G$.  Indeed, we can always take $H$ to be countable (while still accounting for the conjugation action of a dense subgroup of $M$) and give it the discrete topology, and take $G = M \rtimes H$, so we can always choose $\CC_G(M)$ to be countable, even if $M$ is uncountable.  Thus the symmetry $M \cong \CC_G(M)$ that one sees in the finite case is broken in this more general setting.  Indeed, there are likely to be many more examples like Example~\ref{ex:no_complement}(ii), where we have two closed normal subgroups of semisimple type, each the centralizer of the other, such that neither embeds in the other, but both groups embed densely in some other group of semisimple type.

\

\paragraph{\bf Type TW}
We now have $K = \triv$, $S < M$ and $\CC_G(M) = \triv$.  In particular, $H$ does not induce any inner automorphisms of $M$; however, in order for $G/H$ to be irreducible, $\N_H(S)$ must induce a sufficiently large group of automorphisms of $S$ that, for every nontrivial closed $\N_H(S)$-invariant subgroup $T$ of $S$, we have $S \le \cgrp{hTh\inv \mid h \in H}$.  Given the standard homomorphism of $M$ into a direct product, we see that a necessary condition on $T$ here is that $T\CC_M(S)$ must be dense in $M$.  The usual way to build groups of this type is the (unrestricted) twisted wreath product construction introduced by B. H. Neumann in \cite{NeumannWreath}.  We start with a group $H$, a subgroup $H_1$, and a homomorphism $\varphi: H_1 \rightarrow \Aut(S)$ such that $\bigcap_{h \in H}h\varphi\inv(\Inn(S))h\inv = \triv$.  Let $S^H = \prod_H S$, regarded as the set of functions from $H$ to $S$, and let $H$ act on $S^{H}$ by the regular shift.  We then define a base group
\[
B:= \{f \in S^{H} \mid \forall h \in H, k \in H_1: f(hk) = \varphi(k\inv)(f(h))\};
\]
one sees that $B$ is $H$-invariant, so we have a semidirect product $B \rtimes H$ formed as a subgroup of $S^{H} \rtimes H$.  As a group, $B$ is isomorphic to $S^{H/H_1}$, with $H$ permuting the copies of $S$ in the same manner as its action on $H/H_1$, and in particular $H_1$ occurs as the normalizer of one of the copies of $S$, acting via $\varphi$.  Since $\bigcap_{h \in H}h\varphi\inv(\Inn(S))h\inv = \triv$, the action of $H$ on $H/H_1$ is faithful; from there we deduce that $\CC_G(B) = \triv$.  We then see that $H$ is maximal in $B \rtimes H$ if and only if $\varphi(H_1)$ does not preserve any proper nontrivial subgroup of $S$.

If we want to obtain a Polish group at the end of the construction, some allowances and adjustments for the topology need to be made.  We take $S$ to be some topologically simple Polish group.  We want to have at most countably many copies of a simple group in the base, with $H$ having a continuous permutation action on the copies, so take $H$ to be a Polish group with an open subgroup $H_1$ of at most countable index, and require the homomorphism $\varphi: H_1 \rightarrow \Aut(S)$ to yield a continuous action of $H_1$ on $S$.  The end result is then a semidirect product $M \rtimes H$ where $M = \prod_{H/H_1} S$.

The easiest example of a primitive twisted wreath product is $\prod_{\Alt(6)/\Alt(5)}\Alt(5) \rtimes_{\varphi} \Alt(6)$, where $\Alt(6)$ permutes six copies of $\Alt(5)$ with its usual permutation action, but additionally the point stabilizer acts on the corresponding copy of $\Alt(5)$ as its inner automorphism group.  One can easily make nondiscrete Polish analogues of this, for example, the group $\prod_{\Sym(\Nb)/\Sym(\Nb)_0}\Sym(\Nb)_0 \rtimes_{\varphi} \Sym(\Nb)$, where $\Sym(\Nb)$ acts in the usual way on $\Nb$ and $\Sym(\Nb)_0$ is a point stabilizer, acting on the corresponding copy in the base group by conjugation.  There are also \tdlc second-countable examples with compact monolith, for instance if $S$ is a finite simple group, we can embed $S$ in a countable group $H$ in which $\bigcap_{h \in H}hSh\inv = \triv$ (for example a free product $H = S \ast \Zb$) and then form the twisted wreath product $\prod_{H/S}S \rtimes_{\varphi} H$.  The same construction yields discrete examples of type TW, starting from an infinite discrete simple group $S$, by replacing the base group with the direct sum instead of the direct product.  What is more difficult is to find examples of type TW where $G$ and $S$ are locally compact but not compact or discrete.  If we only have finitely many copies of $S$, then $|H:H_1|$ is finite and so it is difficult to arrange for $\varphi\inv(\Inn(S))$ to have trivial core in $H$, while also making $\varphi(H_1)$ large enough to eliminate proper nontrivial $\varphi(H_1)$-invariant closed subgroups of $S$.  On the other hand, if we want to have infinitely many copies of $S$ in $M$, we run into the problem that $\prod_{H/H_1} S$ is not locally compact.  One might be tempted to replace the direct product $\prod_{H/H_1} S$ with some local direct product $\bigoplus_{H/H_1} (S,U)$ where $U$ is a compact open subgroup of $S$, but this causes problems for the action of $H$, because in the action of $\varphi(H_1)$ on $S$, we cannot ask for $U$ to be $S$-invariant (as this would lead to a nontrivial compact $H$-invariant subgroup of the base group).

\

For the remaining types, $M \cap H$ is nontrivial but $M$ has trivial centralizer.  To simplify matters we will only consider the case that $M$ is a $G/H$-supplement.

\paragraph{\bf Type AS}
This is the case where $M = S$ and $\CC_G(M) = \triv$, so $G$ is some group of automorphisms of $S$ containing the inner automorphisms, and $H$ is some maximal closed subgroup of $G$ that does not contain $S$.  There are known infinite examples of this type of coset space, including those coming from the class $\ms{S}$ of nondiscrete compactly generated topologically simple \tdlc groups: for example, by the main theorem of \cite{CapMar}, every maximal parabolic of a simple geometric locally Kac--Moody group is abstractly maximal, and \cite[Theorem 26]{SmithDuke} gives a sufficient condition (leading to $2^{\aleph_0}$ examples) for a group acting on a tree to be compactly generated and simple with a primitive action.  A source of examples where $H$ is not open is when $G$ is a boundary-$2$-transitive subgroup of $\Aut(\mc{T})$, for $\mc{T}$ a regular tree; $H$ is the stabilizer of an end; and $S$ is the cocompact topologically simple open normal subgroup of $G$ provided by \cite[Proposition~3.1.2]{BurgerMozes}.  However, there is not much we can say systematically here about actions of type AS for groups in $\ms{S}$, let alone more general classes of topologically simple groups, as essentially nothing is known in general about their maximal closed subgroups (or even their maximal open subgroups).

\

\paragraph{\bf Type SD}
Let us assume for the moment that $M$ is a direct product of its components, regarded as the set of functions from an indexing set $I$ to $S$, where $2 \le |I| \le \aleph_0$.  For each equivalence relation $P$ on $I$ that is invariant under some transitive subgroup of $\Aut(M)$, we define the block diagonal subgroup
\[
S^*_P := \{f \in M \mid (i,j) \in P \Rightarrow f(i) = f(j)\},
\]
and in particular define the diagonal subgroup $S^* := S^*_{I \times I}$.  Then $S^*$ is a closed subgroup of $M$, and the closed subgroups of $M$ containing $S^*$ are exactly the block diagonal subgroups.  Thus we can build a Polish group $M \rtimes H$ where $S^*$ is maximal among proper $H$-invariant subgroups of $M$ by choosing a Polish group $H$ that acts continuously and faithfully on $M$ and primitively on $I$, in other words, $I$ is isomorphic as an $H$-space to $H/\N_H(S)$ and the point stabilizer $\N_H(S)$ is a maximal subgroup of $H$.  By continuity, $\N_H(S)$ must be open in $H$.  We additionally require that there is a closed embedding $\phi$ of $S^*$ into $H$ that is compatible with the conjugation action of $S^*$ on $M$, so that $\phi(S^*)$ accounts for exactly the elements of $H$ that induce inner automorphisms of $M$.  (The simplest case is to take $H = S^* \rtimes H^*$ where $H^*$ faithfully permutes the components of $M$, but greater generality is possible for finite groups, and probably also for infinite Polish groups.)  We can then form the quotient group $G = (M \rtimes H)/T$ where $T = \{(t\inv, \phi(t)) \mid t \in S^*\}$.  Notice that the quotient map $\pi$ restricts to a topological embedding on both $M$ and $H$, so we can also regard $M$ and $H$ as subgroups of $G$.  With this identification and the way $\phi$ was specified, we see that $\CC_G(M) = \triv$.  Thus $G/H$ is a faithful coset space, with $G/H$-supplement $M \in \mc{M}_G$, and $M \cap H = \pi(S^*) \cong S$.

The above construction gives many examples of faithful irreducible coset spaces in Polish groups, as well as faithful irreducible coset spaces in locally compact groups with a compact minimal normal subgroup.  The coset space $H/\N_H(S)$ is also irreducible.  If we allow $S$ to be locally compact but not compact, we still get locally compact examples when $I$ is finite (in which case the faithful version of $H/\N_H(S)$ is finite and hence described by the O'Nan--Scott theorem).  What is less clear is whether there are versions of this construction where $M$ is not a direct product of its components (for example, a local direct product, which can be locally compact with infinitely many noncompact components).  The difficulty is that beyond the direct product case, it is not clear if we can define a ``diagonal subgroup'' $S^*$ in a way that ensures that $S^*$ accounts for $\Inn(S)$ (or at a minimum, accounts for enough inner automorphisms of $S$ that it does not leave any proper closed nontrivial subgroup of $S$ invariant).

\

\paragraph{\bf Type CD}
This is the compound form of type SD: we first form a direct product $M_0$ of $2 \le \kappa \le \aleph_0$ copies of the topologically simple group $S$, and form the diagonal subgroup $S^*_0$ in the same manner as before.  Then we form a Polish group $M$ containing $M_0$ as a proper closed normal subgroup.  (Here we can take $M$ to be a direct product of copies of $M_0$, but there is more flexibility.)  Suppose now that there is a Polish group $H$ admitting a faithful continuous action on $M$, such that $M$ is the quasi-product of the set $\{h(M_0) \mid h \in H\}$, in other words, distinct $H$-conjugates of $M_0$ commute and the set of all $H$-conjugates of $M_0$ generates $M$ topologically, with $\bigcap_{h \in H}\CC_M(h(M_0)) = \triv$.  We also require that $\N_H(M_0)$ permutes primitively the copies of $S$ inside $M_0$ and that $S^{**} \cap M_0 = S^*_0$ where $S^{**} = \cgrp{h(S^*_0) \mid h \in H}$.    Suppose also that there is a closed embedding $\phi$ of $S^{**}$ into $H$ compatible with the conjugation action of $S^{**}$ on $M$, such that $\phi(S^{**})$ accounts for exactly the elements of $H$ that induce inner automorphisms of $M$.  Then we can form the quotient group $G = (M \rtimes H)/T$ where $T = \{(t\inv, \phi(t)) \mid t \in S^{**}\}$, and embed $M$ and $H$ in $G$ via the quotient map.  As before, we see that $G/H$ is a faithful irreducible coset space, with $M \in \mc{M}_G$; this time, $M \cap H = \pi(S^{**})$ and the group $\pi(S^{**})$ is of semisimple type, being a quasi-product of copies of $S^*_0 \cong S$.

\

\paragraph{\bf Type PA}
The final O'Nan--Scott type is a compound form of type AS.  Suppose that $S/T$ is an irreducible coset space of $S$ (it is faithful because $S$ is topologically simple).  We then form a Polish group $M$ of semisimple type, with components isomorphic to $S$.  (Again, the direct product will always work, but there may be other choices.)  This time, $H$ is a Polish group admitting a faithful continuous action on $M$ by automorphisms, acting transitively on the components, such that $\N_H(S) = \N_H(T)$.  Let $T^* = \bigcap_{h \in H}\N_M(h(T))$; note that $S \cap T^* = T$.  We then require a closed embedding $\phi$ of $T^{*}$ into $H$ compatible with the conjugation action of $T^{*}$ on $M$, such that $\phi(T^*)$ accounts for exactly the elements of $H$ that induce inner automorphisms of $M$.    Similarly to before we form $G = (M \rtimes H)/\grp{(t\inv, \phi(t)) \mid t \in T^{*}}$ and embed $M$ and $H$ in $G$ via the quotient map.  To ensure that $G/H$ is a faithful irreducible coset space, we need only confirm that $\CC_G(M) = \triv$ (we have ensured this in the same way as for types SD and CD) and that $T^*$ is maximal among proper closed $H$-invariant subgroups.  Given a closed $H$-invariant subgroup $M'$ of $M$ properly containing $T^*$, then for some $m \in T'$, $h \in H$ and $t \in T$, we have $[m,h(t)] \not\in h(T)$.  It follows that $\grp{[m,h(t)],h(T)}$ is dense in $h(S)$; since $h(T) \le T^* \le M'$, we have $h(S) \le M'$; and since $M'$ is $H$-invariant, we have $h'(S) \le M'$ for all $h' \in H$ and hence $M' = M$.

\subsection{Open questions}\label{sec:irreducible_que}

The following are known not to occur for primitive permutation groups, but we have not ruled them out for faithful irreducible coset spaces.  It would be particularly interesting to have answers under one of the following additional hypotheses: (a) $G$ is Polish; (b) $G$ is locally compact; (c) $G$ is compactly generated and locally compact.

\begin{quest}\label{que:minimal}
Let $G/H$ be a faithful irreducible coset space.
\begin{enumerate}[(i)]
\item Can $\mc{M}_G$ be infinite?
\item Can we have $|\mc{M}_G| > 2$?
\end{enumerate}
\end{quest}

\begin{quest}\label{que:abelian}
Let $G/H$ be a faithful irreducible coset space and suppose there is $M \in \mc{M}_G$ abelian.
\begin{enumerate}[(i)]
\item Can there exist a nontrivial closed normal subgroup $N$ of $G$ such that $M \cap N = \triv$?
\item Can we have $M < \CC_G(M)$?
\item Can we have $MH < G$?
\end{enumerate}
\end{quest}

Note that the parts in Questions~\ref{que:minimal} and \ref{que:abelian} are ordered so that ``yes'' for one part implies ``yes'' for the later parts.  There is also a relationship between Question~\ref{que:minimal}(i) and Question~\ref{que:abelian}(i) in the case that $G$ is a compactly generated locally compact group without nontrivial compact or discrete normal subgroups.  In that setting, it is known for more general reasons (see \cite{CM}) that all but finitely many $M \in \mc{M}_G$ are abelian, so if $G$ is an example for Question~\ref{que:minimal}(i), then it is also an example for Question~\ref{que:abelian}(i).

The following is also open, and describes a phenomenon which goes against the existing classification results for primitive actions, but has not been ruled out in the present setting.

\begin{quest}
Let $G$ be a compactly generated locally compact group and let $H \le G$ be such that $G/H$ is a nondiscrete faithful irreducible coset space.  Can we have $\mc{M}_G = \emptyset$?
\end{quest}


\begin{thebibliography}{99}

\addcontentsline{toc}{section}{References}

\bibitem{BanksElderWillis}
C. Banks, M. Elder and G. A. Willis, Simple groups of automorphisms of trees determined by their actions on finite subtrees, J. Group Theory 18 (2015), 235--261.

\bibitem{BeckerKechris}
H. Becker and A. Kechris, The Descriptive Set Theory of Polish Group Actions, Cambridge Univ. Press, Cambridge, 1996.

\bibitem{Bourbaki}
N. Bourbaki, El\'{e}ments de math\'{e}matique. Fascicule XXIX. Livre VI: Int\'{e}gration. Chapitre 7: Mesure de Haar. Chapitre 8: Convolution et repr\'{e}sentations, Actualit\'{e}s Scientifiques et Industrielles, No. 1306, Hermann, Paris, 1963.

\bibitem{BurgerMozes}
M. Burger and Sh. Mozes, Groups acting on trees: from local to global structure. {\it Publ. Math. IH\'ES} {\bf 92} (2000), 113--150.

\bibitem{CCMT}
P.-E.~Caprace, Y. de Cornulier, N. Monod and R. Tessera, Amenable hyperbolic groups, J. European Math. Soc. 17 (2015), no. 11, 2903--2947.

\bibitem{CapMar}
P.-E. Caprace and T. Marquis, Open subgroups of locally compact Kac--Moody groups, Math. Z. 274 (2013), no. 1, 291--313.

\bibitem{CM}
P-E. Caprace and N. Monod, Corrigendum to ``Decomposing locally compact groups into simple pieces'' (Math. Proc. Camb. Phil. Soc. 150 (1) (2011) 97--128).  Published online by Cambridge University Press: 21 November 2017. 

\bibitem{CapRad}
P.-E. Caprace and N. Radu, Chabauty limits of simple groups acting on trees, J. Inst. Math. Jussieu 19 (2020), no. 4, 1093--1120.

\bibitem{CaretteDreesen}
M. Carette and D. Dreesen, Locally compact convergence groups and $n$-transitive actions, Math. Z. 278 (2014), no. 3--4, 795--827.

\bibitem{CT}
A. Clau\ss{}nitzer and A. Thom, Aspects of $p$-adic operator algebras, preprint (2019), arXiv:1904.12723.


\bibitem{HilgertNeeb}
J. Hilgert and K.-H. Neeb, Lie Semigroups and their Applications, Springer-Verlag, Berlin Heidelberg, 1993.

\bibitem{MZ}
D. Montgomery and L. Zippin, Topological transformation groups, Interscience Publishers, New York-London, 1955.

\bibitem{NeumannWreath}
B. H. Neumann, Twisted wreath products of groups, Arch. Math. 14 (1963), no. 1, 1--6.

\bibitem{Plat66}
V. P. Platonov, Locally projectively nilpotent subgroups and nilelements in topological groups, Izv. Akad. Nauk SSSR Ser. Mat. 30 (1966), 1257--1274.

\bibitem{Praeger}
C. E. Praeger, Finite quasiprimitive graphs. Surveys in combinatorics, 1997 (London), 65--85, London Math. Soc. Lecture Note Ser., 241, Cambridge Univ. Press, Cambridge, 1997. 


\bibitem{Radu}
N. Radu, A classification theorem for boundary $2$-transitive automorphism groups of trees, Invent. math. 209 (2017), 1--60.

\bibitem{ReidDistal}
C. D. Reid, Distal actions on coset spaces in totally disconnected locally compact groups, J. Topology and Analysis 12 (2020), no. 2, 491--532.

\bibitem{Reid_abelian}
C. D. Reid, A classification of the abelian minimal closed normal subgroups of locally compact second-countable groups, J. Group Theory (2020), https://doi.org/10.1515/jgth-2020-0107

\bibitem{RW-EC}
C. D. Reid and P. R. Wesolek, The essentially chief series of a compactly generated locally compact group, Math. Ann. 370 (2018), no. 1--2, 841--861.

\bibitem{RW-Polish}
C. D. Reid and P. R. Wesolek (with an appendix by F. Le Ma\^{i}tre), Chief factors in Polish groups, preprint (2021), to appear in Math. Proc. Camb. Phil. Soc., arXiv:1509.00719v4

\bibitem{SmithDuke}
S. M. Smith, A product for permutation groups and topological groups. Duke Math. J. 166 (2017), no. 15, 2965--2999.

\bibitem{SmithPrimitive}
S. M. Smith, The structure of primitive permutation groups with finite suborbits and t.d.l.c. groups admitting a compact open subgroup that is maximal, preprint (2019), arXiv:1910.13624

\bibitem{WillisSG}
G. A. Willis, Scale groups, preprint (2020), arXiv:2008.05220.

\end{thebibliography}
\end{document}